\documentclass[10pt,reqno]{amsart}
\usepackage{yfonts}
\usepackage{amsmath}
\usepackage{enumerate}
\usepackage{amsthm}
\usepackage{amssymb}
\usepackage{graphics}
\usepackage[colorlinks]{hyperref}
\newtheorem{theorem}{Theorem}[section]
\newtheorem{lemma}{Lemma}[section]
\newtheorem{proposition}{Proposition}[section]
\theoremstyle{definition}

\theoremstyle{remark}

\newtheorem{remark}{Remark}[section]
\newtheorem{corollary}{Corollary}[section]

\newcommand{\R}{\mathbb R}

\newcommand{\Z}{{\mathbb{Z}}}
\newcommand{\Q}{{\mathbb{Q}}}
 
\begin{document}
\title[Critical points]{Manifolds which admit maps with finitely many critical points 
into spheres of small dimensions}
\author[L.Funar]{Louis Funar}
\address{Institut Fourier,
Laboratoire de Mathematiques UMR 5582,   
Universit\'e Grenoble Alpes,
CS 40700, 
38058 Grenoble, France}
\email{louis.funar@univ-grenoble-alpes.fr} 
\author[C.Pintea]{Cornel Pintea}
\address{Department of Mathematics, "Babe\c{s}-Bolyai" University, 400084 M. Kog\u{a}lniceanu 1,
Cluj-Napoca, Romania}
\email{cpintea@math.ubbcluj.ro}

\begin{abstract}
We construct, for  $m\geq 6$ and $2n\leq m$,  closed manifolds $M^{m}$ with finite nonzero $\varphi(M^{m},S^{n}$),    where $\varphi(M,N)$ denotes the minimum number of critical points of a 
smooth map $M\to N$. We also give some explicit families of examples for  even $m\geq 6, n=3$, taking advantage of the Lie group structure on $S^3$.  Moreover,  there are infinitely many such examples with $\varphi(M^{m},S^{n})=1$. 
Eventually we compute the signature of the manifolds $M^{2n}$ occurring for even $n$.

\noindent MSC Class: 57R45, 57 R70, 58K05.
\end{abstract}
\maketitle

\section{Motivation}
We set $\varphi(M,N)$ for the minimum number of critical points of a smooth map $M\to N$ between compact manifolds, 
which extends the F-category defined and studied by Takens in \cite{Ta}.  
Following the work of Farber (see \cite{Far,FarSchu}) we have: 
\begin{equation}
\varphi(M,S^1)= \left\{\begin{array}{cl}
\varphi(M,\R),&  {\rm if } \;  H^1(M,\Z)=0;\\
0,&  {\rm if }\;  M  \;{\rm fibers \; over } \; S^1;\\
1,& {\rm  otherwise}. 
\end{array}\right. 
\end{equation}
More precisely for any non-zero class $\xi$ in 
$H^1(M,\Z)$ there exists a function $f:M\to S^1$ 
in the homotopy type prescribed  by $\xi$ with at most one  
critical point. This was extended in \cite{FarSchu} 
to closed 1-forms in a prescribed non-zero class in 
$H^1(M,\R)$ having at most one zero. 
The question on whether there is a closed 
non-singular 1-form (i.e. a fibration over $S^1$  
for integral classes) was answered by 
Thurston in dimension 3 (see \cite{Thu}) and Latour 
for $\dim (M)\ge 6$ (see \cite{Lat}). 
Notice that $\varphi(M,\R)\leq \dim M +1$ (see \cite{Ta}). 
 
The aim of this paper is to show that there are examples of manifolds $M^m$ with {\em nontrivial} 
(i.e. finite nonzero) $\varphi(M^{m},S^{\left[\frac{m}{2}\right]-k})$, for $m\geq 6$, $m\geq 2k\geq 0$ where, when present, the superscripts denote the dimensions of the corresponding manifolds and to describe  how to construct all of them for $(m,n)=(6,3)$.

Recall that in  \cite{AF1} the authors found that $\varphi(M^m, N^n)\in\{0, 1,\infty\}$, when $0 \leq  m-n \leq  2$, 
except for the exceptional pairs of dimensions $(m, n)\in\{(2, 2),(4, 3),(4, 2)\}$.
Further, if $m-n = 3$ and there exists a smooth function $M^m \to N^n$ with finitely many critical
points, all of them cone-like, then $\varphi(M^m, N^n)\in \{0, 1\}$ 
except for the exceptional pairs of dimensions $(m, n)\in\{(5, 2),(6, 3),(8, 5)\}$.
On the other hand in \cite{FPZ} the authors provided many nontrivial examples and showed that  
$\varphi(M^m,S^n)$ can take arbitrarily large {\rm even} values for $ m=2n-2$, $n\in\{3,5,9\}$;
these examples  were classified in \cite{Fu} for $n\in\{3,5\}$.

In the first part of the present paper we approach this question by 
elementary methods. In \cite{Fu} the first author outlined a method for constructing manifolds 
with finite  $\varphi(M^6,S^3)$ using generalized {\em Hopf links}, which was further 
detailed in \cite{AHSS}. Our goal is to show that a slight extension 
of this construction provides nontrivial examples  for {\em all dimensions} 
of the form  $(m, \left[\frac{m}{2}\right]-k)$,  where $m\geq 6$, $k\geq 0$, and in particular we can find 
manifolds with $\varphi(M,N)=1$ in this range of dimensions.  In some sense 
these provide other high dimensional analogs of Lefschetz fibrations.  
The simplest approach come from a closed formula computing the Euler characteristic $\chi(M^{2n})$ in terms of the combinatorial 
data used in the construction. We also give some explicit families of examples for  dimensions $(m\geq 6, 3)$, taking advantage of the Lie group structure on $S^3$.  
In particular, we find that $\varphi_c(S^6,S^3)=\infty$, where $\varphi_c$ counts the minimum number of critical points of smooth functions with only cone-like singularities.  
The last part is devoted to computation of signatures which are obstructions to fibration over even dimensional spheres. 
We obtain manifolds with boundary whose signatures are non-zero. 

It would be interesting to know how accurate are our 
estimates  -- compare  with the lower bounds for $\varphi(M^{2n-2}, S^{n})$ obtained in 
\cite{FPZ} --  in order to characterize the set of values taken by $\varphi(M^m,S^n)$. 

Notice that no nontrivial examples are known for $m <2n-2$ and 
the present methods do not apply, though as polynomials maps 
with isolated singularities do exist (\cite{Loo}) for $m-n\geq 4$. 

\section{Constructions of manifolds with finite $\varphi$ and statement of results}

\subsection{Fibered links and local models for isolated singularities}
Recall, following Looijenga (\cite{Loo})  that the isotopy class of the oriented 
submanifold $K=K^{m-n-1}$ of dimension $(m-n-1)$ of $X^{m-1}$ with a trivial 
normal bundle  is called {\em generalized Neuwirth-Stallings fibered} (or $(X^{m-1}, K^{m-n-1})$ 
is a  generalized Neuwirth-Stallings pair) 
if, for some trivialization $\theta: N(K)\to K\times D^n$ of the tubular 
neighborhood $N(K)$ of $K$ in $X^{m-1}$,  
the fiber bundle $\pi\circ \theta: N(K)-K \to S^{n-1}$ admits an extension 
to  a smooth fiber bundle 
$f_K:X^{m-1}- K\to S^{n-1}$. Here $\pi:K\times (D^n-\{0\})\to S^{n-1}$ is the 
composition of the radial 
projection $D^n-\{0\}\to S^{n-1}$ with the second factor projection.
The data $(X^{m-1}, K, f_K, \theta)$ is then called an {\em open book decomposition} with binding $K$ while $K$ is called a {\em fibered link}. 
This is equivalent to the condition that the closure of every fiber is 
its compactification  by the binding link. When $X^{m-1}=S^{m-1}$ we have the classical notions of Neuwirth-Stallings fibrations and pairs. 

Recall now from\cite{Loo,KN,Rud}  that open book decompositions $(S^{m-1}, K, f_K, \theta)$ give rise to 
isolated singularities $\psi_K:(D^m,0)\to (D^n,0)$  by means of the formula:
\[\psi_K(x)=\left\{\begin{array}{cl}
\lambda(\|x\|)f_K\left(\frac{x}{||x||}\right), & {\rm if} \; \frac{x}{||x||}\not\in N(K); \\
\lambda\left( \|x\|\cdot \left\|\pi_2\left(\theta\left(\frac{x}{\|x\|}\right)\right)\right\|\right)f_K\left(\frac{x}{||x||}\right), & {\rm if} \; \frac{x}{||x||}\in N(K); \\
0, & {\rm if }\;  x=0,
\end{array}
\right.
\]
where $\pi_2:K\times D^n\to D^n$ is the projection on the second factor and $\lambda:[0,1]\to [0,1]$ is any smooth 
strictly increasing map sufficiently flat at $0$ and 1 such that $\lambda(0)=0$ and $\lambda(1)=1$. 
If  $K$ is in generic position, namely the space generated by vectors in $\R^m$ with endpoints in $K$ coincides with 
the whole space $\R^m$, then  $(d\psi_K)_{0}=0$, i.e. $\psi_K$ has rank $0$ at the origin. 
We then call such $\psi_K$ {\em local models} of isolated singularities. 

Looijenga in \cite{Loo} proved that a Neuwirth-Stallings pair $(S^{m-1},L^{m-n-1})$  can be realized by a {\em real polynomial} map  
if $L$ is invariant and the open book fibration $f_L$ is equivariant with respect to the antipodal maps.  
In particular, the connected sum $(S^{m-1},K)\sharp ((-1)^{m}S^{m-1}, (-1)^{m-n}K)$ is a Neuwirth-Stallings pair isomorphic to the 
link of a real polynomial isolated singularity $\psi_K:(\R^m,0)\to (\R^n,0)$.

\subsection{Cut and paste  local models}\label{cutpaste}
We can glue together a patchwork of such local models 
to obtain maps $M^{m}\to N^{n}$ with finitely many 
critical points.   
Let $\Gamma$ be a bicolored decorated graph with vertices of two colors. 
Each black vertex $v$ of $\Gamma$ is decorated by  a fibered link $L_v^{m-n-1}$ of $S^{m-1}$. To every vertex 
$v$ there is associated an open book fibration $f_{L_v}:S^{m-1}-N(L_{v})\to S^{n-1}$ which extends to a 
smooth local model map with one critical point $\psi_v=\psi_{L_v}:D_v^{m}\to D^n$. Its generic fibers are called local fibers. 
Each white vertex $w$ is labeled by some $(m-n)$-manifold $F(w)$ 
whose boundary has as many connected components as the  degree of $w$. 

If there are no white vertices, then we glue together the disks $D_v$ using the pattern of the graph $\Gamma$ by identifying 
one component of $N(L_{v})$ to one component 
of $N(L_{w})$  if $v$ and $w$ are adjacent in $\Gamma$. 
The identification has to respect the trivializations  
$N(L_{v})\to D^n$  and hence one can take them to be 
the same as in the double construction. Note that $N(L_v)=L_v\times D^{n}$ and thus identifications respecting the trivialization 
correspond to homotopy classes $[L, {\rm Diff}(D^n,\partial)]$.

Otherwise, we glue together the disks $D_v$  and $F(w)\times D^n$  along part of their boundaries using the 
pattern of the  graph $\Gamma$. One identifies a component of $N(L_{v})$  to a component of 
$\partial F(w)\times  D^n$ whenever there is 
an edge between  $v$ and $w$, such that the two trivializations of these manifolds 
do agree and the fibers of the open book and of the trivial fibration glue together.  In such a case $\partial F(w)$ and the link $L_v$ 
should have the same number of components.  When $F(v)$ is a bunch of cylinders we recover the former construction. We 
then obtain  a manifold with boundary $X(\Gamma)$
endowed with a smooth map $f_{\Gamma}: X(\Gamma)\to D^n$, whose 
singularities correspond to the black vertices.  

The restriction of $f_{\Gamma}$ to the boundary is a locally trivial $F$-fibration over $S^{n-1}$. Let now 
$\Gamma_1,\Gamma_2,\ldots,\Gamma_p$ be a set of bicolored decorated graphs whose associated fibrations  
are {\em cobounding}, namely such that there exists a fibration over 
$S^n\setminus \sqcup_{i=1}^{p}D^n$, generally not unique, extending the boundary 
fibrations restrictions $\psi_{\Gamma_i}:=f_{\Gamma_i}\big|_{\partial X(\Gamma)}$, $1\leq i\leq p$. 
Any such set $\Gamma_1,\Gamma_2,\ldots,\Gamma_p$ determines  
a closed manifold $M(\Gamma_1,\Gamma_2,\ldots,\Gamma_p)$
endowed with a map with finitely many critical points into $S^n$.

In particular, we can realize the double 
of ${f}_{\Gamma}$ by gluing together ${f}_{\Gamma}$ and its mirror image. 
We could also generalize this to maps taking values into 
an arbitrary closed manifold $N^n$.

\subsection{Constructions of fibered links in dimensions $(2n,n)$, $n\geq 3$}\label{half}
Let us recall the construction from \cite{Fu,AHSS}. 
It is known that for $n\geq 3$ there is only 
one embedding of $S^{n-1}$ in $S^{2n-1}$.  The situation undergoes only  
little changes in the case of links. By Haefliger's  classification theorem 
(see \cite{Hae2a,Hae2b})  the link $L=\sqcup_{j=0}^d S^{n-1}_j$ is 
uniquely determined, up to isotopy,  by its linking matrix ${\rm lk}_L$ and we denote it as $L_{{\rm lk} L}$. 
Note that the diagonal entries of ${\rm lk}_L$ are not defined and by convention we set 
them $0$. 

The {\em generalized Hopf links} with $d+1\geq 2$ components are those links $L=\sqcup_{j=0}^d S^{n-1}_j$ 
for which the spheres $S^{n-1}_1, \ldots S^{n-1}_{d}\subset S^{2n-1}$ 
are Hopf duals  to a fixed  {\em preferred} $S^{n-1}_0\subset S^{2n-1}$, namely their linking number 
$lk(S_0,S_j)=\pm 1$, for $j\geq 1$.  We will suppose that $lk(S_0,S_j)=1$, for $j\geq 1$, in the sequel, so that 
the most important information is the linking sub-matrix ${\rm lk}_{L^{\circ}}$ of the sub-link  $L^{\circ}=\sqcup_{j=1}^d S^{n-1}_j$.
Denote by $\widetilde{A}$ a  $(-1)^n$-symmetric matrix obtained from a $d\times d$ matrix $A$ by adding a 
first line and a first column of $1$s with $0$ on the diagonal.  
 
One observed in \cite{Fu} that for every integral  
$(-1)^n$-symmetric $d\times d$ matrix $A$ with 
trivial diagonal the link $L_{\widetilde{A}}$  
has the property that 
its complement $S^{2n-1}\setminus N(L_{\widetilde{A}})$ naturally fibers over $S^{n-1}$. 
The fibers of this fibration are holed disks which intersect transversally every component 
$S_j^{n-1}$, with $j\geq 1$ in one point, while their closure contain $S_0^{n-1}$. 
Note that this fibration come along with a trivialization of the boundary: $\partial N(S_0^{n-1})$ is foliated by  
preferred longitudinal spheres while $\partial N(S_j^{n-1})$, for $j\geq 1$, are foliated by  
preferred meridian spheres.

The fibration of $S^{2n-1}\setminus N(L_{\widetilde{A}})$ does not satisfy the last condition in the definition of a Neuwirth-Stallings pair. Although a link is  always fibered if its complement fibers (not necessarily as 
an open book decomposition)  when $n=2$, by a suitable change of the framing,  this is not so in higher dimensions. 
However, there is a simple way to convert transversal intersections of the fiber 
with $S_j^{n-1}$ into one of binding type  by doing surgery.  
Specifically,  we denote by $X_A^{2n-1}$ the result of gluing together $S^{2n-1}\setminus N(L_{\widetilde{A}})$ and  $(d+1)$ solid tori $S^{n-1}\times D^{n}\sqcup_{j=1}^d D^n\times S^{n-1}$ such that: 
\begin{enumerate}
\item for $j=0$ the solid torus $S^{n-1}\times D^{n}$ is glued along $\partial N(S_0^{n-1})$ 
such that $S^{n-1}\times \{pt\}$ correspond with the preferred longitude spheres;
\item for $j\geq 1$ the $j$-th  copy of the solid torus $D^n\times S^{n-1}$ is glued along  $\partial N(S_j^{n-1})$  
such that $\{pt\}\times S^{n-1}$ correspond to the preferred meridian spheres.  
\end{enumerate}
The cores of the newly attached solid tori form a $(d+1)$-components link 
$K_A^{n-1}=\sqcup_{0}^dS^{n-1}\subset X^{2n-1}_A$. Though as $X_A^{2n-1}$ might not be a sphere in general, 
$(X_A^{2n-1}, K_A)$ is a generalized Neuwirth-Stallings pair. 
Note that the link complements $X_A\setminus N(K_A)$ and $S^{2n-1}\setminus L_{\widetilde{A}}$ are diffeomorphic and the corresponding fibrations match each other. Thus the fibers of the corresponding 
open book fibration $f_{K_A}:X_A\setminus K_A\to S^{n-1}$ are still holed disks. 
We warn the reader that the notions of longitude/meridian spheres do not correspond for the two link complements.

When $X^{2n-1}_A$  is diffeomorphic to $S^{2n-1}$ we obtain  
a classical Neuwirth-Stallings pair  $(S^{2n-1}, K_A)$. 
Furthermore, $X_A^{2n-1}$ is homeomorphic to a sphere $S^{2n-1}$
if and only if $A$ is unimodular, i.e. $\det A=\pm 1$ (see \cite{AHSS}). 
This provides already examples of fibered links $K_A$ in those dimensions when there are no exotic spheres, for instance when $n=3$. Moreover, when  $n=3$ every fibered link over $S^2$ is isotopic to some $K_A$ (see \cite{AHSS,Fu}) since their fibers should be simply connected and hence holed disks.  
This is equally true for $n>3$ if we restrict ourselves to those links whose components are spheres. However, when $n >3$ links of isolated singularities might be non-simply connected links.

Furthermore, since the connected sum $X_A^{2n-1}\sharp \overline{X_A^{2n-1}}$
is diffeomorphic to $S^{2n-1}$ for any $n$ one obtained in (\cite{AHSS}, Corollary 4.2) that the links of the form $K_{A\oplus -A}\subset S^{2n-1}$ are fibered for any $n >3$, if $A$ is unimodular. Notice that the number of components  in this construction satisfies 
$d\equiv 1 ({\rm mod } \; 4)$.  Further, we also have $\sharp_{\theta_{2n-1}}X^{2n-1}$ is diffeomorphic to $S^{2n-1}$, where 
$\theta_{2n-1}$ denotes the order of the group of homotopy spheres in dimension $(2n-1)$.
The connected sum construction due to Looijenga (\cite{Loo}) shows that $K_{\oplus_{1}^{\theta_{2n-1}}A}$ 
is fibered, for any $n>3$, when $A$ is unimodular.

We can therefore use fibered links of the form $K_A^{n-1}\subset S^{2n-1}$, which will be called {\em generalized Hopf links} in the sequel. The cut and paste procedure from section \ref{cutpaste} then produces  manifolds with boundary $X^{2n}(\Gamma)$ 
endowed with maps $\psi_{\Gamma}:X^{2n}(\Gamma)\to D^n$ with finitely many critical points.   
The generic fiber of $\psi_{\Gamma}$  is  $\sharp_{g}S^1\times S^{n-1}$, where $g$ is the rank of 
$H_1(\Gamma)$. If one allows orientation-reversing gluing homeomorphisms then  one could also obtain  non-orientable 
fibers  homeomorphic to a twisted $S^{n-1}$-fibration over the circle. 

The restriction of $\psi_{\Gamma}$ to the boundary is a 
$\sharp_{g}S^1\times S^{n-1}$-fibration over $S^{n-1}$. Let 
$\Gamma_1,\Gamma_2,\ldots,\Gamma_p$ be a set of graphs associated to 
a family of cobounding fibrations, 
namely such that there exists a fibration over 
$D^n\setminus \sqcup_{i=1}^{p-1}D^n$ extending the boundary 
fibrations restrictions of  $f_{\Gamma_i}$, $1\leq i\leq p$. 
We remark that $H_1(\Gamma_i)$ should be isomorphic.
Then we can glue together $\psi_{\Gamma_j}$  to obtain some 
manifold $M(\Gamma_1,\ldots, \Gamma_p)$ endowed 
with a smooth map with finitely many critical points onto $S^n$.

When $n=3$, all 6-manifolds $M^6$ admitting a smooth 
map $M^6\to S^3$ with finitely many 
cone-like singularities arise by this  construction.

\subsection{Fibered links in dimensions $(2n+1,n)$, where $n\geq 2$}
We can construct a much larger family of examples from existing ones, by means of 
a method used by Looijenga in \cite{Loo} to construct non-trivial local isolated singularities. 
Specifically, we consider the {\em spinning} of Hopf 
links,  in a similar manner as the spinning of a knot. 
Consider  a  link $L=\sqcup_{j=0}^d S_j^{n-1}\subset S^{2n-1}$ with a choice of one component $S_i^{n-1}$ to be spin off. 
We isotope  $L$ so that  all components but 
$S_i^{n-1}$ lie in the interior of the upper half space $H_+^{2n-1}:=\{(x_1,\ldots,x_{2n-1})\in\R^{2n-1}; \ x_{2n-1}\geq 0\}$,  while the intersection of $S_i^{n-1}$ with the lower half space consists of a  hemisphere. 
We now spin $H_+^{2n-1}$ in $\mathbb{R}^{2n}$ around $\mathbb{R}^{2n-2}$ so that each point $(x_1,\ldots,x_{2n-1})\in H_+^{2n-1}$ sweeps out the circle $(x_1,\ldots,x_{2n-2},x_{2n-1}\cos\theta,x_{2n-1}\sin\theta)$, $\theta\in [0,2\pi]$. 
The spinning orbits of the hemisphere along $\sqcup_{j\neq i}S_j^{n-1}$  form  a link of the form 
$SL=S_i^{n}\sqcup_{j\neq i}(S^1\times S_j^{n-1})\subset S^{2n}$  (for more details on the knot counterpart see \cite{Fr}). 
When $L$ is a fibered link, the spinning links $SL$ are all fibered. 
In particular this is the case when $L=K_A$. If $F^n=S^n\setminus \sqcup_{j=0}^dD_j^n$ is the fiber of $L$ then 
$SF^{n+1}= S^{n+1}\setminus (D_i^{n+1}\sqcup_{j\neq i}^dS^1\times D_j^n)$ is the fiber of $SL$.

Note that we can iterate this procedure $k$ times and  by choosing each time the same 
spinning component  we obtain links of the form  $S^{n+k-1}\sqcup_1^d(S^1)^k\times S^{n-1}\subset S^{2n+k-1}$.

\subsection{Fibered links in dimensions $(2n,k)$ and $(2n+1,k)$, where $n\geq k\geq 2$}\label{subhalf}
The {\em rank} of a critical point is the rank of the differential at that point. 
Given a smooth map $\psi:(D^m,0)\to (D^k,0)$, $k\geq 2$  with an isolated singularity at $0$ of rank zero 
we consider the map $\Pi \psi:(D^m,0)\to (D^{k-1},0)$ 
obtained by composing $\psi$ with the projection $\Pi:D^k\to D^{k-1}$. This is again a smooth map with an isolated 
singularity at the origin of rank zero. 

According to \cite{J,AD} the local Milnor fiber $F_{\Pi \psi}$ of $\Pi \psi$ around $0$ is homeomorphic to $F_{\psi}\times [0,1]$, if $\psi$ 
is a real polynomial.   

Starting from a smooth map  $\psi_L:(D^{2n},0)\to (D^n,0)$ as constructed in section \ref{half} out of a generalized Hopf link $L$ in generic position
we deduce by iterated projections  smooth maps with an isolated  singularity at the origin 
$\Pi^k\psi:(D^{2n},0)\to (D^{n-k},0)$ in all dimensions $(2n,n-k)$, with $0\leq k\leq n-1$. Links $K^{n+k-1}\subset S^{2n-1}$ 
obtained from these maps will be called {\em generalized Hopf links} in dimensions $(2n,n-k)$.

Assume $\psi_L$ is the local model associated to a fibered generalized Hopf link $L$ with $(d+1)$ components in generic position. 
Then the local fiber $F_{\psi_L}$ is diffeomorphic to a  $n$-disk with $d$ handles of  index $(n-1)$  attached along trivially embedded and 
unlinked spheres  $S^{n-2}\subset \partial D^n$.

The link $L_{\Pi \psi_L}\subset S^{2n-1}$ associated to $\Pi\psi_L$   
is the union of local fibers $f_L^{-1}(\overline{\Pi}^{-1}(0))$, where $\overline{\Pi}:S^{n-1}\to D^{n-1}$ is the projection.  Now $\overline{\Pi}^{-1}(0)=\{n,s\}$ 
is a pair of points, the north and the south pole of $S^{n-1}$ with respect to the projection $\overline{\Pi}$. Therefore 
$L_{\Pi \psi_L}$ is the closure of the union of the two local fibers  
$f_L^{-1}(n)$ and $f_L^{-1}(s)$ of $f_L$, i.e. their union with $L$. 

The link $L_{\Pi \psi_L}\subset S^{2n-1}$ associated to $\Pi\psi_L$ is 
\[
\begin{array}{lll}
S^{2n-1}\cap(\Pi\psi_L)^{-1}(0) & =S^{2n-1}\cap\psi_L^{-1}\big(\Pi^{-1}(0)\big)=S^{2n-1}\cap\psi_L^{-1}\big([sn]\big)\\
& = \left[(S^{2n-1}\setminus N(L))\cap f_L^{-1}\big([sn]\big)\right]\cup\left[N(L)\cap\psi_L^{-1}\big([s0)\cup\{0\}\cup(0n]\big)\right]\\
& = \left[(S^{2n-1}\setminus N(L))\cap f_L^{-1}\big(\{s,n\}\big)\right]\cup L\cup\left[N(L)\cap\psi_L^{-1}\big([s0)\cup(0n]\big)\right],
\end{array}
\]
as $\psi_L\big|_{S^{2n-1}\setminus N(L)}=f_L\big|_{S^{2n-1}\setminus N(L)}$. 
Note that $\psi_L\big|_{N(L)}\neq f_L\big|_{N(L)}$ as 
$\psi_L(L)=0$ while $f_L(L)\subseteq S^{n-1}$.  Since $N(L)\cap\psi_L^{-1}\big([s0)\big)$ is homeomorphic with 
$N(L)\cap f_L^{-1}(s)$ and $N(L)\cap\psi_L^{-1}\big((0n]\big)$ is homeomorphic with $N(L)\cap f_L^{-1}(n)$ we obtain that the link 
$L_{\Pi \psi_L}\subset S^{2n-1}$ associated to $\Pi\psi_L$  is homeomorphic with the closure of the union of the two local fibers  
$f_L^{-1}(n)$ and $f_L^{-1}(s)$ of $f_L$, i.e. their union with $L$. 

Furthermore the open book fibration $f_{L_{\Pi \psi}}: S^{2n-1}\setminus L_{\Pi \psi}\to S^{n-2}$ 
is obtained as $f_{L_{\Pi \psi}}(x)= R \overline{\Pi} f(x)$, 
where $R: D^{n+1}\setminus \{0\}\to S^{n-2}$ is the radial projection. 
If $x\in S^{n-2}$ let $\gamma_x\subset S^{n-1}$ be the great arc 
passing through $n, s$ and $x=\overline{\Pi}^{-1}(x)\in S^{n-1}$. 
Then the local fiber $F_{\Pi \psi_L}$ of $\Pi \psi_L$ is 
the union of fibers $f_L^{-1}(\gamma_x)$. It follows that 
$F_{\Pi \psi_L}$ is homeomorphic to $F_{\psi_L}\times [0,1]$.

By induction the local fiber of $\Pi^kf$ is a $(n+k)$-disk 
with  $d$ handles of index  $(n-1)$ 
attached along trivially embedded and unlinked $S^{n-2}\subset \partial D^{n+k}$. 
It follows that the local fiber $F_{\Pi^kf}={\sharp_{\partial}}_{d}  S^{n-1}\times D^{k+1}$, where 
${\sharp_{\partial}}$ denotes the boundary connected sum of manifolds with boundary. 
In particular the corresponding link $L_{\Pi^kf}\subset S^{2n-1}$ is diffeomorphic to a connected sum $\sharp_{j=1}^{d} S^{n-1}\times S^{k}$.  Note that  the link $L_{\Pi^kf}$ is connected when $k\geq 1$.

It follows that  for $k\geq 1$ any decorated graph $\Gamma$ which can occur in the construction above 
consists of two black vertices and an edge joining them or else a single white vertex connected to several black vertices. 
Note that the gluing map in the former case is highly not unique, the result depending on the 
corresponding element of mapping class group of  $\sharp_{j=1}^{d} S^k\times S^{n-1}$.

\subsection{Statement of results}
Our first result shows that all these examples are nontrivial:

\begin{theorem}\label{generalodd}
Let $\Gamma_1, \Gamma_2,\ldots, \Gamma_p$ be bicolored graphs decorated by generalized Hopf links in dimensions $(2n,n-k)$ 
as in section \ref{subhalf} such that the fibrations $f_{\Gamma_1}, f_{\Gamma_2}\ldots, f_{\Gamma_p}$ cobound.  
When $n-k$ is even we assume that the total number $s$ of black vertices of the graphs 
$\Gamma_1, \Gamma_2,\ldots, \Gamma_p$ is odd. We have then the inequalities: 
 \begin{equation}
1\leq \varphi(M^{2n}(\Gamma_1, \Gamma_2,\ldots, \Gamma_p), S^{n-k})\leq s. 
\end{equation}
\end{theorem}

\begin{remark} 
The fibrations with fiber $F$ over $S^{n-1}$, $n\geq 3$,  are classified by their 
characteristic elements in the group $\pi_{n-2}({\rm Diff}(F))$. A collection 
of fibrations cobound if the sum of their characteristic elements is trivial. 
This provide abundant examples verifying the assumptions of the theorem for odd $n-k$. 
Notice that for even $n-k$,  it is not clear that  there exists a collection 
$\Gamma_1, \Gamma_2,\ldots, \Gamma_p$ of bicolored decorated graphs
with odd total number of vertices in order to be able to use theorem \ref{generalodd} to finding non-trivial examples. 
\end{remark}

Let now  $\varphi_c$ count the minimum number of 
critical points of smooth maps with only {\em cone-like} singularities (see \cite{KN}). 

\begin{theorem}\label{63}
If $\varphi_c(M^6,S^3)$ is finite nonzero then $M$ is diffeomorphic 
to $M^6(\Gamma_1,\ldots, \Gamma_p)$, for some decorated bicolored graphs $\Gamma_i$.   
In particular $\pi_1(M)$ is a (closed) 3-manifold group. 

Moreover, if  $\pi_1(M^6)=1$ and  $\chi(M)\geq -1$ then  either $\varphi_c(M^6,S^3)=0$, or 
$\varphi_c(M^6, S^3)=\infty$. 
\end{theorem}

Since $S^6$ does not fiber over $S^3$ (see e.g. \cite{AF1}) we derive:
\begin{corollary}
We have $\varphi_c(S^6,S^3)=\infty$. 
\end{corollary}

We think that it is possible to classify all manifolds $M^6$ with  finite $\varphi_c(M^6,S^3)$. 

We further show that this method could indeed produce explicit 
examples with $\varphi$ equal to one, in all dimensions. We state our result below separately for odd and even dimensions, 
as the combinatorial data is slightly different.  

\begin{theorem}\label{generaleven}
Suppose that  $n\geq 3$  and  the decorated graph is as follows:    
\begin{enumerate}
\item for $k=0$ a tree $\Gamma_0$ with one black vertex  decorated by a generalized Hopf link and several white vertices decorated by disks.
\item for $k\geq 1$, the graph $\Gamma_0$ has a single black vertex $v$ 
decorated by a generalized Hopf link   $L_{\Pi^kL}$, where $L$ is a  $(n-1)$-dimensional 
generalized Hopf link with $d+1\geq 5$ components and a white vertex, the two vertices being connected by an edge. 
The white vertex $w$ is decorated by  $F_w={\sharp_{\partial}}_{d}D^{n}\times S^{k}$. 
\end{enumerate} 
Then  
\[\varphi(M^{2n}(\Gamma_0), S^{n-k}) =1.\]
\end{theorem}
\begin{theorem}\label{odd} 
Suppose that  $n\geq 3$  and the decorated graph is as follows: 
\begin{enumerate}
\item for $k=0$ the graph $\Gamma_0$ is a tree consisting of one black vertex 
decorated by the fibered link $SK_A$ which is adjacent to $d+1\geq 2$ white vertices, one of which 
being decorated by the disk $D^{n+1}$ 
and the remaining white vertices being decorated by $S^1\times D^n$. 
\item for $k\geq 1$ the graph $\Gamma_0$ has a single black vertex $v$ decorated by $L_{\Pi^kSL}$, where $L$ is 
an  $(n-1)$-dimensional generalized Hopf link with $d+1\geq 5$ components
and a white vertex, the two vertices being connected by an edge. 
The white vertex $w$ is decorated by  the manifold 
$F_w=\left({\sharp_{\partial}} _{j=1}^{d}D^{n}\times S^{k+1}\right)\sharp_{\partial}\left({\sharp_{\partial}} _{j=1}^{d}S^{k}\times D^{n+1}\right)$. 
\end{enumerate}  
Then 
\[\varphi(M^{2n+1}(\Gamma_0), S^{n}) =1\]
\end{theorem}
The gluing map between the decoration and the local fiber associated to the black vertex will be specified in the proof. 

The only drawback of this method is that  we don't have 
an explicit description of the manifolds of the form  
$M^{m}(\Gamma_1,\ldots, \Gamma_p)$. Using different tools we can provide a first sample of 
easy to understand examples in arbitrary high dimensions, which might be interesting by themselves, as follows: 

\begin{proposition}\label{products}
We have 
\[ 1\leq \varphi(S^4\times S^4\times \cdots \times S^4,S^3)\leq 2^m,\]
when we have $m$ factors $S^4$. 
Moreover, we have 
\[1\leq \varphi\left((\sharp_{r_1}S^2\times S^2)\times(\sharp_{r_2}S^2\times S^2)
\times \cdots \times (\sharp_{r_m}S^2\times S^2),S^3\right)\leq 2^m(r_1+1)\cdots(r_m+1).\]
\end{proposition}
The existence of the Hopf fibration $S^3\to S^2$ implies:  
\begin{corollary} We have 
\[\varphi\left((\sharp_{r_1}S^2\times S^2)\times(\sharp_{r_2}S^2\times S^2)
\times \cdots \times (\sharp_{r_m}S^2\times S^2),S^2\right)\leq 2^m(r_1+1)\cdots(r_m+1).\]
\end{corollary}
When $m=1$, $r_1=1$ the left hand side vanishes. 
It seems that otherwise it is positive.

\begin{corollary}\label{evendim}
There exist examples with nontrivial 
$\varphi(M^{2n},S^3)$, for every $n\geq 2$.
\end{corollary}
This is a consequence of Theorem \ref{generalodd}  and the proof of Proposition \ref{products}.

The second part of this paper aims at a deeper understanding of these examples when  $n$ is even and, in particular, 
to approach the case when $n-k$ is even in theorem \ref{generalodd}. 

A necessary condition for $M^{2n}$ to admit a fibration over $S^k$ is that $\chi(M^{2n})=0$, when $k$ is odd and 
$\chi(M^{2n})\equiv 0 \; ({\rm mod} \; 2)$, for even $k$.  When $n$ is even there are stronger requirements for a manifold to be a fibration over $S^n$. Recall that the {\em signature} of the compact oriented $M$ is set to be zero unless its 
dimension is multiple of $4$, in the later case being the signature of the symmetric 
bilinear form on the middle dimension cohomology 
given by cup product evaluated on the fundamental class.   
A classical theorem due to Chern, Hirzebruch and Serre  (\cite{CHS}) states that 
whenever we have a fibration $E\to B$ with fiber $F$ of oriented 
compact manifolds  such that the action of $\pi_1(B)$ on the cohomology 
$H^*(F)$ is trivial, then the signature is multiplicative, namely 
\[\sigma(M)=\sigma(B)\sigma(F).\] 
In particular this happens when 
$\pi_1(B)$ is trivial. This is known not to be true for general fibrations 
as for instance in the case of the Atiyah-Kodaira fibrations (see \cite{At,Kod}), which are fibrations of some 4-manifolds of signature 256 over surfaces. 
In particular, if  $\sigma(M)\neq 0$, then $\varphi(M,S^p)\geq 1$, for any $p$, thus also for even values of $p$.

Our next goal is the explicit computation of $\sigma(M(\Gamma_1,\Gamma_2,\ldots,\Gamma_p))$. 
Observe that for even $n$ we have 
$\sigma(M^{2n}(\Gamma_1,\Gamma_2,\ldots,\Gamma_p))\equiv s \; ({\rm mod} \, 2)$.

\begin{theorem}\label{signeven}
Consider $n$ even. There exist  graphs $\Gamma$  decorated by generalized Hopf links in dimensions $(2n,n-k)$ 
as in section \ref{subhalf} such that 
\[ \sigma(M(\Gamma))\neq 0.\] 
\end{theorem}

\section{Proofs of Theorems \ref{generalodd}, \ref{63}, \ref{generaleven} and \ref{signeven}}\label{proofs}
\subsection{Preliminaries on fibered generalized Hopf  links in dimensions $(2n,n)$}
Denote by $K_i$, $0\leq i\leq d$, the components of $K_A$, 
which are indexed as the components of $L_{\widetilde{A}}$.  
Note that unlike arbitrary fibered links $K_A$ also have a {\em canonical framing} in $X_A$, namely a set of isotopy classes of 
parallel copies $K_i^{\sharp}\subset \partial N(K_i)$ obtained by intersecting the generic fiber of the 
given open book decomposition with the boundary of the link complement.
In particular, it makes sense to consider the diagonal of the linking matrices of $K_A$ whose entries 
are $lk(K_i^{\sharp}, K_i)$.  We can actually identify the link $K_A$ when $A$ is unimodular, in the Lemma below.
\begin{lemma}\label{KvsL}
If $A$ is unimodular then $K_A=L_{A^*}$, where the linking matrix in the canonical framing $A^*$ is the $(-1)^n$-symmetric matrix with entries: 
\[ A^*_{ij}=\left \{\begin{array}{ll}
(A^{-1})_{ij}, & {\rm if } \; 1\leq i, j \leq d;\\
-\sum_{k=1}^d (A^{-1})_{kj}, & {\rm if } \; i=0,  1\leq j\leq d;\\
\sum_{k=1}^d \sum_{l=1}^d(A^{-1})_{kl}, & {\rm if } \; i=j=0\\
\end{array}\right.
\]
\end{lemma}
\begin{proof}
Let $X_s$ denote the result of filling all but the $s$-th boundary components using surgery as above. 
Then $X_s$ is $(n-2)$-connected and the Mayer-Vietoris sequence reads: 
\[ H_{n-1}(\sqcup_{j=0, j\neq s}^d K_i\times \partial D^{n})\to H_{n-1}(S^{2n-1}\setminus N(L_{\widetilde{A}}))
\oplus H_{n-1}(\sqcup_{j=0, j\neq s}^d K_i\times D^{n})\to H_{n-1}(X_s)\to 0\] 
If $X_A$ is homeomorphic to a sphere $H_{n-1}(X_s)\cong \Z$ and then 
the linking number $lk(K_j, K_s)$ in $X_A$ is the image of the class of $K_j$ in $\Z$.  
Moreover, $H_{n-1}(S^{2n-1}\setminus N(L_{\widetilde{A}}))\cong \oplus_{j=0}^d\Z\mu_j$, where the classes 
$\mu_j$ correspond to the meridians spheres around each boundary component. Let $\delta_j$ denote the generator 
of $H_{n-1}(K_j\times D^n)$.   We give $K_j$ the orientation induced as a boundary component of the fiber (which disagrees 
with the convention in \cite{AHSS}).  

If $s\neq 0$, then it follows (see the computations from \cite{AHSS}, proof of Lemma 3.4) that we have the presentation: 
\[ H_{n-1}(X_s)=\frac{\oplus_{j=0}^d\Z\langle \mu_j\rangle \oplus_{i=0, i\neq s}^d\Z\langle \delta_i\rangle}{\Z\langle\delta_0+\sum_{j=1}^d \mu_j\rangle 
\oplus_{1\leq i\leq d, i\neq s}\Z\langle\mu_i-\delta_i\rangle\oplus\Z \mu_0\oplus_{1\leq i\leq d, i\neq s}\Z\langle \sum_{j=1}^dA_{i j}\mu_j\rangle}\]
Further the homomorphism $ev:  H_{n-1}(X_s)\to \Z$ given on the generators by 
\[ ev(\mu_i)=(A^{-1})_{is}, 1\leq i\leq d, \;  \;  ev(\mu_0)=0 \]
\[ ev(\delta_i)=(A^{-1})_{is}, 1\leq i\leq d, \;  i\neq s, \; ev(\delta_0)=-\sum_{i\neq s} (A^{-1})_{is}\]
is  well defined, and it is an isomorphism since $A$ is invertible over $\Z$. 
The  class of $K_j$ and respectively $K_s^{\sharp}$ in $H_{n-1}(X_s)$ is represented by $\mu_j$, if $j\neq s$,  and hence 
\[ lk(K_j, K_s)=A^*_{js}, j\neq 0, lk(K_s^{\sharp}, K_s)=A^*_{ss}\]
Further, the class of $K_0$ is represented by $-\sum_{j=1}^d\mu_j$ and hence 
\[ lk(K_0, K_s)=-\sum_{j=1}^dA^*_{js}=A^*_{0s}\]
If $s=0$, we have a similar presentation of $H_{n-1}(X_0)$:
\[ H_{n-1}(X_0)=\frac{\oplus_{j=0}^d\Z\langle \mu_j\rangle \oplus_{i=1}^d\Z\langle \delta_i\rangle}
{\oplus_{1\leq i\leq d}\Z\langle\mu_i-\delta_i\rangle\oplus_{1\leq i\leq d}\Z\langle \mu_0+\sum_{j=1}^dA_{i j}\mu_j\rangle}\] 
Further the homomorphism $ev:  H_{n-1}(X_0)\to \Z$ given on the generators by 
\[ ev(\mu_i)=-\sum_{j=1}^d(A^{-1})_{ij}, 1\leq i\leq d, \;  \;  ev(\mu_0)=1 \]
\[ ev(\delta_i)=-\sum_{j=1}^d(A^{-1})_{ij}(A^{-1})_{is}, 1\leq i\leq d\]
is also an isomorphism. We derive: 
\[ lk(K_j, K_0)=-\sum_{i=1}^d(A^{-1})_{ji}=A^*_{j0}, j\neq 0\]
\[ lk(K_0^{\sharp}, K_0)=\sum_{j=1}^d\sum_{k=1}^d(A^{-1})_{jk}=A^*_{00}\]
\end{proof}
\subsection{Proof of Theorem \ref{generalodd}}
We only need to prove that $M(\Gamma_1,\ldots,\Gamma_p)$ 
does not fiber over $S^{n-k}$. For the sake of simplicity of exposition 
we will only consider the case where there are no insertion of trivial fiber bundles here and hence we can drop the 
decoration. Note that this implies that $\Gamma_i$ only contain black vertices and 
that there are no univalent vertices of $\Gamma_i$.

Now the Euler characteristic $\chi$ 
is multiplicative in locally trivial fiber bundles, namely for  any locally trivial  fibration  
$\pi:E\to B$ with fiber $F$ we have $\chi(E)=\chi(B)\chi(F)$. 
This is well-known to hold in the case when the action of $\pi_1(B)$ on the cohomology 
$H^*(F)$ is trivial, in particular when $\pi_1(B)=0$. The standard argument to prove this uses spectral sequences. 
Nevertheless, the multiplicativity of the Euler characteristic holds in full generality, 
as soon as $E,F$ and $B$ are finite CW complexes, by induction on the number of cells of the basis. 
This is obviously true when $B$ has only one cell, in which case $E$ is a product. Assume that the multiplicativity is true 
for fiber bundles over CW complexes with at most $N$ cells, and consider a complex $B$ with $N+1$ cells. Let $e^n$ be a $n$-cell of $B$. The restriction $\pi^{-1}(B-{e^n})\to B-e^n$ is a fiber bundle so that 
$\chi(\pi^{-1}(B-{e^n}))=\chi(B-{e^n})\chi(F)$.  By excision we have 
$H^*(E,\pi^{-1}(B-{e^n}))=H^*(e^n\times F, \partial e^n\times F)$. This implies that 
$\chi(E,  \pi^{-1}(B-{e^n}))=(-1)^n\chi(F)$ and hence 
$\chi(E)=\chi(\pi^{-1}(B-{e^n})+\chi(E,  \pi^{-1}(B-{e^n}))=\chi(B)\chi(F)$. This proves the induction step.

Thus a necessary condition for a space $E$ to fiber over the $S^{n-k}$, is that 
$\chi(E)=0$, if $n-k$ is odd and $\chi(E)\equiv 0 \;({\rm mod} \, 2)$, when 
$n-k$ is even, respectively.

One can compute 
$\chi(M(\Gamma_1,\ldots,\Gamma_p))$ using the local picture description of each singularity. 

Consider first the case when $k=0$. 
A critical point associated to a vertex 
of $\Gamma_i$ of valence $(d+1)$ comes with a local model 
whose link has $(d+1)$ components. 
As in the case of Lefschetz fibrations 
we obtain the local model from a  
fibration over the punctured disk $D^n-\{0\}$ with fiber 
$D^n-\sqcup_{i=1}^{d} D_i^n$ by  adjoining 
one singular fiber over $0$ which is the cone over 
the boundary. This amounts to adjoin to the trivial 
fibration over $D^n$ a number of $d$ 
handles of index $n$, corresponding to 
crushing the vanishing cycle 
$\bigvee_{d} S_i^{n-1}$ to a point.   
This handlebody description can be turned into 
a cell-decomposition, and therefore each local model 
corresponds to a fibration with $d$ cells of dimension 
$n$ adjoined. 
Gluing together all local models by the patchwork 
explained in the introduction produces a block $X(\Gamma_j)$ obtained from a fibration over $D^n$ with 
$t_j$ cells of dimension $n$ added, where $t_j=2m_j-s_j$, 
$m_j$ being the total number of edges in the $\Gamma_j$ and 
$s_j$ the total number of vertices. Since each vertex has 
valence at least $2$ we have $m_j-s_j \geq 0$. An alternative argument is to observe that 
$X(\Gamma_j)$ deformation retracts onto the singular fiber, which is obtained from the the regular fiber by contracting the 
attaching $(n-1)$-spheres corresponding to the $n$-handles above. This shows that 
the dimension of the co-kernel of $H_n(\partial X(\Gamma_j))\to H_n(X(\Gamma_j))$ equals $t_j$. 
    
Therefore 
\[\chi(M(\Gamma_1,\Gamma_2,\ldots, \Gamma_p))=\chi(S^n)\chi(\sharp_{g}S^1\times S^{n-1})
+(-1)^nt =-g((1+(-1)^n)^2+(-1)^nt,\]
where $t$ is the sum of all $t_j$.  
When $n$ is odd $\chi(M(\Gamma_1,\Gamma_2,\ldots, \Gamma_p))=-t\neq 0$ and hence 
it cannot be a fibration over some $n$-manifold. When $n$ is even 
$\chi(M(\Gamma_1,\Gamma_2,\ldots, \Gamma_p))\equiv -t\equiv s\not\equiv 0 \; ({\rm mod} \, 2)$, and thus 
it cannot fiber over $S^n$.  

Let now consider the case $k\geq 1$, by analyzing first the local picture. 
The link of each critical point is now connected. However, there exists 
a collection of disjoint embedded spheres $S^{n-1}$ embedded in the 
local fiber $F_{\Pi^k\psi_L}$, which is diffeomorphic to 
a $(n+k)$-disk  with $d$ copies of $(n-1)$-handles attached to it.
The singular fiber is then homeomorphic to a cone over the link. 
Therefore a regular neighborhood of the singular fiber is 
homeomorphic to the result of attaching $d$ handles of index $n$  
to the regular neighborhood of a generic fiber.   
This description permits to use the arguments above for $k=0$. 
We conclude as above.

\begin{remark}
When singular points arise from the fibered links above, each critical point $s$ contributes 
with $\chi(F_s)-1$ to $\chi(M)$, where $F_s$ is the local fiber around $s$. This holds more generally 
for all fibered links. On the other hand if dimensions  were of the form $(2n+1, k)$, then 
local fibers should verify $\chi(F_s)=1$, according to \cite{ADD,J,MS}. This shows that the contribution 
of every critical point is trivial in odd dimensions and hence the previous arguments cannot work.  
\end{remark}

\subsection{Proof of Theorem \ref{63}}
Every open book fibration $S^5-N(K)\to S^2$ has a simply connected 
fiber $F^3$ (see e.g. \cite{AHSS}). By Perelman's solution to the Poincar\'e Conjecture 
$F^3$ is a disk with holes, and thus $K$ is a disjoint union of spheres 
$S^2$. Therefore $K$ is a generalized Hopf link $L_Q$, for some  matrix $Q$. 
Moreover $L_Q$ is fibered if and only if $Q=\widetilde{A}$, where $A$ is unimodular, according to \cite{AHSS}. 
Thus  for any smooth map $f:M^6\to S^3$ with finitely many 
cone-like critical points there are neighborhoods of the critical points 
to which the restriction of $f$ is equivalent to some local model. Outside these neighborhoods 
the restriction of $f$ should be a locally trivial fiber bundle. Therefore 
$M^6$ is diffeomorphic to some 
$M^6(\Gamma_1,\Gamma_2,\ldots,\Gamma_p)$, 
where $\Gamma_i$ are bicolored decorated graphs and 
$f$ arises as above. We suppose that $M^6$ is not a fibration over $S^3$.  
Every graph $\Gamma_i$ has at least one black vertex, as otherwise we could 
remove it. Each  decorated graph $\Gamma_i$ determines 
$f_{\tilde{\Gamma_i}}: X^6(\Gamma_i)\to D^3$, whose generic fiber is some closed 3-manifold 
$F$, which is independent on $i$. 

Notice that the union $V$ of singular fibers of $f$ is a CW complex 
of dimension 3 embedded in $M$, so that 
$\pi_1(M^6-V^3)\to \pi_1(M^6)$ is an isomorphism. 
The long exact sequence in homotopy associated to the fibration
$f|_{M-V}$ implies that $\pi_1(F^3)\to \pi_1(M^6)$ is surjective, with kernel free abelian.  
Let $F_{ij}$ and $D_{ij}=D^3-\sqcup_{s=1}^{n_{ij}} 
D_s^3$ denote the 3-manifolds with boundaries which occur as labels of the 
white vertices and black vertices, respectively, 
of the graph $\Gamma_i$.  The key point is that local fibers $D_{ij}$ are simply connected. Then  
the generic fiber $F$ is obtained from the  (graph) connected sum 
of $F_{ij}$ and $D_{ij}$. The block $X^6(\Gamma_i)\setminus V$ is the union
of fibered pieces $D^6_v\setminus V$ associated to black vertices $v$  and 
$F_{ij}\times (D^3\setminus\{0\})$ associated to decorated white vertices. 
Moreover we glue together two such adjacent pieces along  the submanifold 
$N(L_{\widetilde{A(v)}})\setminus L_{\widetilde{A(v)}}$, which is simply connected, by transversality. 
Also $\pi_1(D^6_v\setminus V)=1$. Then Van Kampen's theorem implies that 
the inclusion of $F$ into $X^6(\Gamma_i)\setminus V$ induces an isomorphism at the level of fundamental groups, 
and hence $\pi_1(X^6(\Gamma_i)\setminus V)$ is isomorphic to $\pi_1(F)\cong *_{j} \pi_1(F_{ij}) *\mathbb F_{r}$, 
where $r$ is the rank of $H_1(\Gamma_i)$. 
We obtain $M^6\setminus V$ by first gluing together several blocks $X^6(\Gamma_i)\setminus V$
along neighborhoods of boundary fibers and second gluing to the result a trivial fibration $F\times D^3$  
along the whole boundary $F\times S^2$.   Further use of Van Kampen's theorem show that the inclusion of $F$ 
into $M$ is also an isomorphism at the fundamental group level.

Every black vertex $v$ of some $\Gamma_i$ has associated a link 
of the form $L_{\widetilde{A}}$, where  $A$ is unimodular (see \cite{AHSS}). 
But unimodular skew-symmetric matrices have to be of even size, so that every black vertex $v$ has odd degree. 
Assume that $\pi_1(M)$ has not a free factor, so that $r=0$. 
Then each  $\Gamma_i$ should have only one black vertex, since  otherwise the valence of a black vertex being  odd 
it would be at least $3$ and this would produce a free factor in $\pi_1(F)$. The local fiber associated to this black vertex is $D^3\setminus\sqcup_{s=1}^dD_s^3$. 
 Each $F_{ij}$  must have one boundary component; if some $F_{ij}$ had at least two boundary components 
then gluing the local fiber $D^3\setminus\sqcup_{s=1}^dD_s^3$ would produce a free factor in $\pi_1(F)$. 
Thus  the generic fiber $F$ of $f$ is diffeomorphic to $\sharp_{s=1}^dF_s$. 

Suppose now that $\pi_1(M)=1$. Then $F$ is simply connected and hence, by Perelman it is diffeomorphic to $S^3$. 
Moreover, each $F_{ij}$ is diffeomorphic to a disk. 
The computation of the Euler characteristic from the previous section gives us 
\[ \chi(M^6(\Gamma_1,\Gamma_2,\ldots, \Gamma_p))=-\sum_{i=1}^{p}d_i\] 
where $1+d_i\geq 3$ is the degree of the black vertex of $\Gamma_i$.  
In particular, if $\pi_1(M)=1$ and $\chi(M)\geq -1$ then $\varphi_c(M,S^3)=\infty$, as we supposed that 
$M^6$ does not fiber. 
This proves the claim.

\subsection{Proof of Theorem \ref{generaleven}} We need first the following: 
\begin{lemma}\label{trivialfibration} 
In dimensions $(2n,n)$ let the graph $\Gamma_0$ be a tree consisting of one black vertex 
decorated by the fibered link $K_A$ which is adjacent to $d+1\geq 2$ white vertices decorated by disks $D^n$. 
The gluing maps  correspond to the decomposition of $S^n$ as union of two smooth disks $D^n$ along an equatorial sphere. 
Then $\partial X^{2n}(\Gamma_0)$ is diffeomorphic to $S^{n}\times S^{n-1}$ and 
the boundary fibration $\psi_{\Gamma_0}:\partial X^{2n}(\Gamma_0)\to S^{n-1}$ is trivial. 
\end{lemma}
\begin{proof}
We obtain $\partial X(\Gamma_0)$ by doing surgery on the link $K_{A}$, namely gluing to 
$S^{2n-1}\setminus N(K_A)$ the disjoint union of  $(d+1)$ solid tori $\sqcup_{j=0}^d D^n\times S^{n-1}$ such that
the $j$-th  copy of the solid torus $D^n\times S^{n-1}$ is glued along  $\partial N(K_j^{n-1})$  
such that $\partial D^n\times \{pt\}$ correspond to the preferred longitude spheres.  
This is the same as doing surgery on $L_{\widetilde{A}}$ corresponding to the framings given by the longitude 
around $S_0$ and the meridian spheres along $S_j$, $j\geq 2$. Surgery along meridian spheres restores the sphere 
$S^{2n-1}\setminus N(S_0)$, while surgery along the longitude of $S_0$ yields $S^{n-1}\times S^{n}$. 
The fibration structure of $\psi_{\Gamma_0}$ corresponds then to the projection onto $S^{n-1}$. 
\end{proof}

\begin{lemma} \label{trivialfibration2} 
In dimensions $(2n,n-k)$ with $k\geq 1$ let the graph $\Gamma_0$ with a single black vertex $v$ 
decorated by $L_{\Pi^kL}$, where $L$ is a $(n-1)$-dimensional generalized Hopf link with $d+1\geq 5$ components
and a white vertex connected by an edge. 
The white vertex $w$ is decorated by  $F_w={\sharp_{\partial}} _{j=1}^{d}D^{n}\times S^{k}$.
The gluing along $\partial F_w$ corresponds to surgery of the  core $k$-dimensional spheres 
 and hence the global fiber $F$ is  
diffeomorphic to $S^{n+k}$. Then 
$\partial X(\Gamma_0)$ is diffeomorphic to $S^{n+k}\times S^{n-k-1}$ and the 
boundary fibration $\partial X(\Gamma_0)\to S^{n-k-1}$ is trivial. 
\end{lemma}
\begin{proof}
We have the decomposition 
\[\partial X(\Gamma_0)=(S^{2n-1}\setminus ((\sharp_{j=1}^{d}S^{n-1}\times S^{k})\times D^{n-k})) \cup 
({\sharp_{\partial}} _{j=1}^{d}D^{n}\times S^{k})\times S^{n-k-1}\] 
along $\partial (E_v)_k= (\sharp_{j=1}^{d}S^{n-1}\times S^{k})\times S^{n-k-1}$. 
The result follows for $k=0$ from Lemma \ref{trivialfibration}. We use further induction on $k$. 
We add the subscript $k$ to all objects defined so far. 
If the claim holds for $k$ then $(E_v)_k\subset S^{n+k}\times S^{n-k-1}$ and the projection 
$f_k:(E_v)_k\to S^{n-k-1}$ is the restriction of the second factor projection. 
Note that $(E_v)_{k+1}=(R\overline{\Pi})^{-1}(S^{n-k-1}\setminus (D^{n-k-1}(s)\cup D^{n-k-1}(n)))$, where 
$D^{n-k-1}(s)$ and $D^{n-k-1}(n)$ are two disk neighborhoods 
of the north and the south
poles $n,s$ of $S^{n-k-1}$. Then the fibration $f_{k+1}: (E_v)_{k+1}\to S^{n-k-2}$  is the 
composition $(f_k)|_{(E_v)_{k+1}}: (E_v)_{k+1}\to  S^{n-k-1}\setminus (D^{n-k-1}(s)\cup D^{n-k-1}(n))$ 
with the projection $R\overline{\Pi}: S^{n-k-1}\setminus (D^{n-k-1}(s)\cup D^{n-k-1}(n))\to S^{n-k-2}$. 
Further $(E_v)_{k+1}$ is a subfibration of the product fibration 
\[ S^{n+k}\times S^{n-k-1}\setminus (D^{n-k-1}(s)\cup D^{n-k-1}(n)) \to S^{n-k-2}\] 
which itself is a subfibration of $S^{n+k+1}\times S^{n-k-2}\to S^{n-k-2}$. 

It remains to observe that during the process of filling the fibration  $(E_v)_{k+1}$ we adjoined along the boundary 
$((F_w)_{k+1}\times [0,1])\times S^{n-k-2}$, namely $(F_w)_{k+1}\times S^{n-k-2}$. This proves the induction claim. 
\end{proof}

If  $k=0$ we consider the manifold $M(\Gamma_0)=X(\Gamma_0)\cup_{\partial X(\Gamma_0)} S^{n}\times D^n$. 
First $\pi_1(M(\Gamma_0))=0$ and further by Mayer-Vietoris 
$H_j(M(\Gamma_0))=0$, for $1\leq j\leq 2n-1$, $j\neq n$ and $H_n(M(\Gamma_0))=\Z^{d+2}$.  
Thus $M(\Gamma_0)$ is $(n-1)$-connected.

Assume  that $M(\Gamma_0)$ fiber over $S^n$ with fiber $F^n$.  
Then the long exact sequence of the fibration shows that $F^n$ must be $(n-2)$-connected. 
Further, the Wang sequence  yields first:
\[\to  H_{n}(F)\to  H_{2n-1}(M)\to H_{n-1}(F^n)\to H_{2n-2}(F^n)\to H_{2n-2}(M)\to \]
and thus $H_{n-1}(F^n;\Q)=0$, as $n\geq 3$ and second: 
\[\to  0=H_1(F)\to H_{n}(F)\to  H_{n}(M)\to H_{0}(F^n)\to H_{n-1}(F^n)=0\]
and hence $H_0(F)$ has rank $d$, contradiction, thereby proving the claim. 

When $k\geq 1$ we consider $M(\Gamma_0)=X(\Gamma_0)\cup_{\partial X(\Gamma)} S^{n+k}\times D^{n-k}$.
It follows that $\pi_1(M(\Gamma_0))=0$ and by Mayer-Vietoris 
$H_j(M(\Gamma_0))=0$, for $1\leq j\leq 2n-1$, $j\not\in \{n-k,n+k\}$, while 
$H_n(M(\Gamma_0))=\Z^{d}$, $H_{n-k}(M(\Gamma_0))=H_{n+k}(M(\Gamma_0))=\Z$. 
Assume  that $M(\Gamma_0)$ fibers over $S^{n-k}$ with fiber $F^{n+k}$.  
Then the long exact sequence of the fibration shows that $F$ is connected and simply connected. 
The Wang exact sequence 
\[ H_{q}(F)\to  H_{q}(M(\Gamma_0))\to H_{q-n+k}(F)\to H_{q-1}(F)\to H_{q-1}(M(\Gamma_0))\to \]
for $q=2n-1, 2n-2,\cdots, n+k+2$ yields    
\[ H_{n+k-1}(F)=H_{n+k-2}(F)=\ldots =H_{2k+2}(F)=0\]
Further, for $q=n+k+1$ we obtain the exact sequence
\[ 0=H_{n+k+1}(M(\Gamma_0))\to H_{2k+1}(F)\to H_{n+k}(F)\to  H_{n+k}(M(\Gamma_0))\to H_{2k}(F)\to H_{n+k-1}(F)=0\]
which implies that 
\[{\rm rk} H_{2k+1}(F;\Q)={\rm rk} H_{2k}(F;\Q)=u\in\{0,1\}\]
 
If $n\geq 2k+2$ then we can consider $q=n+k-1,\ldots, n$ and derive
\[ H_{2k-1}(F)=H_{2k-2}(F)=\cdots =H_{n}(F)=0\]
From  the exactness of 
\[ 0= H_{n}(F)\to H_{n}(M(\Gamma_0))\to H_{k}(F)\to H_{n-1}(F) \]
we obtain ${\rm rk} H_k(F;\Q)\geq {d}$. But $H_n(F;\Q)=0$, as $n\geq 2k$ and this contradicts the 
Poincar\'e duality for $F^{n+k}$.

If $n=2k+1$ 
then from the exact sequence 
\[ H_{2k+1}(F)\to H_{2k+1}(M(\Gamma_0))\to H_{k}(F)\to H_{2k}(F)\to H_{2k}(M(\Gamma_0))\]
we derive that  both the kernel and the cokernel of the map $H_{2k+1}(M(\Gamma_0);\Q)\to H_{k}(F;\Q))$ 
has rank at most ${\rm rk}(H_{2k+1}(F;\Q)\leq 1$. 
This implies that ${\rm rk}H_{k}(F;\Q)\geq {d}-2$. 
But ${\rm rk}H_{n}(F;\Q)\leq 1$,  from above and this contradicts the Poincar\'e duality for $F^{n+k}$.

If $2k\geq n$ let $a$ be the smallest positive integer such that $a(n-k-1)\geq 2k-n$. 
By using induction and the Wang sequence we obtain, for all natural $m\leq a$ that we have:   
\[ H_{n+k-m(n-k-1)}(F)=\cdots =H_{2k+2-m(n-k-1)}(F)=0\]
\[ {\rm rk}H_{2k+1-m(n-k-1)}(F;\Q)={\rm rk}H_{2k-m(n-k-1)}(F;\Q)=u\in\{0,1\}\]

By letting $q=n+k-1,\cdots, n\geq 2k-a(n-k)$ we derive again that 
${\rm rk}H_n(F;\Q)\leq 1$ while ${\rm rk} H_k(F;\Q)\geq{d}-2$, contradiction. 
This shows that $M(\Gamma_0)$ cannot fiber over $S^{n-k}$.

{\subsection{Proof of Theorem \ref{odd}}
\begin{lemma}\label{trivialfibration3} 
In dimensions $(2n+1,n)$ let the graph $\Gamma_0$ be a tree consisting of one black vertex 
decorated by the fibered link $SK_A$ which is adjacent to $d+1\geq 2$ white vertices, one of which decorated by the disk $D^{n+1}$ and the remaining white vertices are decorated by $S^1\times D^n$. The gluing maps correspond to the decomposition of 
$S^{n+1}$ into the union of two disks $D^{n+1}$ along an equatorial sphere. 
Then 
$\partial X^{2n+1}(\Gamma_0)$ is diffeomorphic to $S^{n+1}\times S^{n}$ and 
the boundary fibration $\psi_{\Gamma_0}:\partial X^{2n+1}(\Gamma_0)\to S^{n}$ is trivial. 
\end{lemma}
\begin{proof}
Let assume that the component $S_0^{n-1}$ of $K_A=L_{A^*}$ is spun. We consider the link 
$L_{\widetilde{A}}=\sqcup_{j=0}^d S_j^{n-1}\subset S^{2n-1}$ as the boundary of a holed disk. The spun component 
$S_0^n$ inherits a longitude by spinning the one of $S_0^{n-1}$, while the other components $S^1\times S_j^{n-1}$   
inherit well-defined meridians by taking their product with $S^1$.  
We obtain $\partial X(\Gamma_0)$ by doing surgery on the link $SK_{A}$, namely gluing to 
$S^{2n}\setminus N(SK_A)$ the disjoint union 
$D_0^n\times S^{n}\sqcup_{j=1}^d S^1\times S^{n-1}\times D_j^n$ such that the $j$-th copy of $S^1\times S^{n-1}\times D_j^n$ is glued along $\partial N(K_j^{n})$ and $D_0^n\times S^{n}$ is glued along $\partial N(K_0^{n})$. 
Surgery along $\partial N(K_j^{n})$ identifies $\partial D_0^n$ with the longitude of $S_0^n$ and 
$S^1\times \partial D_j^n$ with the meridian of $S^1\times S_j^{n-1}$. 
By completing the last surgeries one restores the sphere $S^{2n}\setminus N(S_0)$, while surgery along the longitude 
of $S_0^n$  yields $S^{n+1}\times S^{n}$.  
The fibration structure of $\psi_{\Gamma_0}$ corresponds then to the projection onto $S^{n}$. 
\end{proof}
We consider the manifold $M(\Gamma_0)=X(\Gamma_0)\cup_{\partial X(\Gamma_0)} S^{n+1}\times D^{n}$. 
First $\pi_1(M(\Gamma_0))=0$ and further by Mayer-Vietoris 
$H_j(M(\Gamma_0))=0$, for $1\leq j\leq 2n$, $j\not\in\{n,n+1\}$, 
$H_n(M(\Gamma_0))=H_{n+1}(M(\Gamma_0))=\Z^{d+1}$. Thus $M(\Gamma_0)$ is $(n-1)$-connected. 

Assume  that $M(\Gamma_0)$ fibers over $S^n$ with fiber $F^{n+1}$.  
Then the long exact sequence of the fibration shows that $F^{n+1}$ must be $(n-1)$-connected and the induced map 
$\pi_n(S^n)\to \pi_{n-1}(F)$ is surjective so that $\pi_{n-1}(F)\cong H_{n-1}(F)$ has rank at most $1$. 
Further, the Wang sequence  yields first:
\[\to  H_{2n}(F)\to  H_{2n}(M(\Gamma_0))\to H_{n}(F)\to H_{2n-1}(F^{n+1})\to H_{2n-1}(M(\Gamma_0))\to \]
and thus $H_{n}(F^{n+1};\Q)=0$, as $n\geq 3$ and second: 
\[\to  0=H_1(F)\to H_{n}(F)\to  H_{n}(M(\Gamma_0))\to H_{0}(F)\to H_{n-1}(F)\to H_{n-1}(M(\Gamma_0))=0\]
and hence $H_0(F)$ has rank at least $(d+1)$, which  contradicts the $(n-1)$-connectedness of $F$.  
This proves the claim.

\begin{lemma}\label{trivialfibration4} 
Consider the dimensions $(2n+1,n-k)$, $k\geq 1$ and the graph $\Gamma_0$ with a single black vertex $v$ 
decorated by $L_{\Pi^kSL}$, where $L$ is a $(n-1)$-dimensional generalized Hopf link with $d+1\geq 5$ components
and a white vertex connected by an edge. 
The white vertex $w$ is decorated by  $F_w=\left({\sharp_{\partial}} _{j=1}^{d}D^{n}\times S^{k+1}\right)\sharp_{\partial}\left({\sharp_{\partial}} _{j=1}^{d}S^{k}\times D^{n+1}\right)$.
The gluing along $\partial F_w$ is the one arising in surgery of the  $(k+1)$ and $k$-dimensional  core spheres   and the global fiber $F$ is then 
diffeomorphic to $S^{n+k+1}$. Then 
$\partial X(\Gamma_0)$ is diffeomorphic to $S^{n+k+1}\times S^{n-k-1}$ and the 
boundary fibration $\psi_{\Gamma_0}:\partial X(\Gamma_0)\to S^{n-k-1}$ is trivial. 
\end{lemma}
\begin{proof}
If $L$ is fibered and $L=\partial F^n$, where the fiber $F^n=D^n\setminus\sqcup_{j=1}^dD_j^n$, then 
$SL$ is fibered and its associated fiber is $SF^{n+1}=D^{n+1}\setminus\sqcup_{j=1}^d S^1\times D_j^n$. 
We obtain $SF^{n+1}$ from $D^{n+1}\setminus\sqcup_{j=1}^d D_j^{n+1}$ by adjoining for each 
boundary component $\partial D_j^{n+1}$ one $(n-1)$-handle along a trivially embedded $S^{n-2}\subset S^n$.  
Therefore $SF^{n+1}$ is obtained from $D^{n+1}$ by first adding $d$ handles of index $n$ and further $d$ handles 
of index $(n-1)$, as above. The attaching spheres bound disjoint disks and hence $SF^{n+1}$ is the 
boundary connected sum of $d$ copies of the corresponding result for $d=1$, the later being $D^2\times S^{n-1}\setminus D_0^{n+1}$. Further $L_{\Pi^kSL}$ is fibered with fiber $SF_{(k)}^{n+k+1}=SF^{n+1}\times D^k$, which has the same
description of handles addition along $D^{n+k+1}$ as above.  
We obtain 
$SF_{(k)}^{n+k+1}=
{\sharp_{\partial}} _{j=1}^{d} D^{2+k}\times S^{n-1} {\sharp_{\partial}} _{j=1}^{d} S^n\times D^{k+1}$. 
Note that  $L_{\Pi^kSL}= {\sharp} _{j=1}^{d} S^{1+k}\times S^{n-1} {\sharp} _{j=1}^{d} S^n\times S^{k}$, 
for $k\geq 1$, in particular it is connected. 

The global fiber of $X(\Gamma_0)$ is the union of $SF_{(k)}^{n+k+1}$ and $F_w$. 
The gluing is the connected sum of gluings occurring in the following two spheres decompositions:  
$D^{2+k}\times S^{n-1} \cup S^{1+k}\times D^{n}=S^{n+k+1}$ and 
$S^n\times D^{k+1}\cup D^{n+1}\times S^{k}=S^{n+k+1}$ and thus the global fiber is $S^{n+k+1}$. 
 
The triviality of the $S^{n+k+1}$-fibration $\partial X(\Gamma_0)\to S^{n-k}$ follows by induction on $k$, as above.  
\end{proof}

Let now  $k\geq 1$ and consider $M(\Gamma_0)=X(\Gamma_0)\cup_{\partial X(\Gamma)} S^{n+k+1}\times D^{n-k}$.
It follows that $\pi_1(M(\Gamma_0))=0$ and by Mayer-Vietoris 
$H_j(M(\Gamma_0))=0$, for $1\leq j\leq 2n-1$, $j\not\in \{n-k,n,n+1,n+k+1\}$, while 
$H_n(M(\Gamma_0))=H_{n+1}(M(\Gamma_0))=\Z^{d}$, $H_{n-k}(M(\Gamma_0))=H_{n+k+1}(M(\Gamma_0))=\Z$. 
Thus $M(\Gamma_0)$ is $(n-k-1)$-connected. 
Assume  that $M(\Gamma_0)$ fibers over $S^{n-k}$ with fiber $F^{n+k+1}$.  
Then the long exact sequence of the fibration shows that $F$ is connected and simply connected. 
The Wang exact sequence 
\[ \to H_{q}(F)\to  H_{q}(M(\Gamma_0))\to H_{q-n+k}(F)\to H_{q-1}(F)\to H_{q-1}(M(\Gamma_0))\to \]
for $q=2n, 2n-1,\ldots, n+k+3$ yields  inductively:  
\[ H_{n+k}(F)=H_{n+k-1}(F)=\ldots =H_{2k+3}(F)=0\] 
Further, by taking $q=n+k+2$ we find that 
\[0=H_{n+k+2}(M(\Gamma_0))\to H_{2k+2}(F)\to H_{n+k+1}(F)\to H_{n+k+1}(M(\Gamma_0))\to H_{2k+1}(F)\to H_{n+k}(F)=0\]
Therefore  
\[{\rm rk} H_{2k+2}(F;\Q)={\rm rk} H_{2k+1}(F;\Q)=u\in \{0,1\}\]
If $n+1\geq 2k+3$ then we can consider $q=n+k,n+k-1,\ldots, n+1$ and derive 
\[ H_{2k}(F)=H_{2k-1}(F)=\cdots =H_{n+1}(F)=0\]
From the exactness of: 
\[ 0= H_{n+1}(F)\to H_{n+1}(M(\Gamma_0))=\Z^d\to H_{k+1}(F)\to H_{n}(F)\to H_{n}(M(\Gamma_0))\]
we obtain that ${\rm rk}H_{k+1}(F;\Q)\geq {d}$. 
But ${\rm rk}H_{n}(F;\Q)\leq 1$, as $n\geq 2k+2$ and this contradicts the Poincar\'e duality for $F^{n+k+1}$.  

If $n=2k+1$, then from the exact sequence 
\[ H_{2k+2}(F)\to H_{2k+2}(M(\Gamma_0))\to H_{k+1}(F)\to H_{2k+1}(F)\to H_{2k+1}(M(\Gamma_0))\]
we derive that  both the kernel and the cokernel of the map $H_{2k+2}(M(\Gamma_0);\Q)\to H_{k+1}(F;\Q))$ 
has rank at most ${\rm rk} H_{2k+2}(F;\Q)\leq 1$. 
This implies that ${\rm rk}H_{k+1}(F;\Q)\geq {d}-2$. 
But ${\rm rk}H_{n}(F;\Q)\leq 1$,  from above and this contradicts the Poincar\'e duality for $F^{n+k+1}$. 

If $2k\geq n$  let $a$ be the smallest positive integer such that $a(n-k-1)\geq 2k+1-n$. 
By using induction and the Wang sequence we obtain, for all natural $m\leq a$ that we have:   
\[ H_{n+k-m(n-k-1)}(F)=\cdots =H_{2k+3-m(n-k-1)}(F)=0\]
\[ {\rm rk}H_{2k+2-m(n-k-1)}(F;\Q)={\rm rk}H_{2k+1-m(n-k-1)}(F;\Q)=u\in\{0,1\}\]

By letting $q=n+k,\cdots, n+1 \geq 2k+1-a(n-k-1)$ we derive as above that 
 ${\rm rk}H_{n}(F;\Q)\leq 1$ while ${\rm rk}H_{k+1}(F;\Q)\geq {d}-2$, contradiction. 
This shows that $M(\Gamma_0)$ cannot fiber over $S^{n-k}$. 

\subsection{Proof of Proposition \ref{products}}
Let $F:X\rightarrow Y$ be a differentiable map. We denote by  
$d_xF:T_x(X)\rightarrow T_{F(x)}(Y)$ its differential 
at $x\in X$. 
If $(G,\cdot)$ is a Lie group, the left and right translations 
by $g\in G$ are the maps $L_g:G\rightarrow G$, $L_g(z)=gz$ and 
$R_g:G\rightarrow G$, $R_g(z)=zg$, respectively.
 Smooth maps $A:M\rightarrow G$, $B:N\rightarrow G$  have 
a well-defined multiplication by setting  
$A\odot B:M\times N\longrightarrow G$, 
$(A\odot B)(z,w)=A(z)B(w)$.

\begin{lemma}\label{lieproducts}
Let $M^m$, $N^n$ be smooth manifolds and $(G,\cdot)$ be a Lie group
of dimension $\dim G\leq \min(m,n)$. 
For any smooth maps 
$A:M\rightarrow G$, $B:N\rightarrow G$ we have:  
\begin{equation}
C\left(A\odot B\right)\subseteq C(A)\times C(B).
\label{eq10.02.13.2}
\end{equation}
\end{lemma}
\begin{proof}
We first need the following   easy formula 
(for the particular case $M=N=G$ and $A=B=id_{_G}$ see \cite[p. 42]{Michor}):
\begin{equation}
[d_{(x,y)}(A\odot B)](u,v)=(d_{B(y)}L_{A(x)})\left(d_{y}B(v)\right)+
(d_{A(x)}R_{B(y)})\left(d_xA(u)\right)
\label{eq10.02.13.1}
\end{equation}
for all $(u,v)\in T_x(M)\times T_y(N)\cong T_{(x,y)}(M\times N)$.
\label{prop10.02.13.1}
This implies that the image of $d(A\odot B)]_{(x,y)}$
is the subspace 
\[(d_{B(y)}L_{A(x)})\left(d_{y}B\left(T_y(N)\right)\right)+
(d_{A(x)}R_{B(y)})\left(d_{x}A\left(T_x(M)\right)\right)\]
If $(x,y)\in (M\times N)\setminus(C(A)\times C(B))$ then 
either $x$ is a regular point for $A$ or $y$
is a regular point for $B$. 
By symmetry we may assume that $x$ is a regular point of $A$  
and hence, by our assumptions on the dimensions, 
$(d_{x}A)\left(T_x(M)\right)=T_{B(y)}(G)$. 
Then, by the formula above 
the range of $d(A\odot B)_{(x,y)}$ contains 
\[(d_{A(x)}R_{B(y)})(d_{x}A)\left(T_x(M)\right)=
(d_{A(x)}R_{B(y)})(T_{A(x)}(G))\]
Since $R_{B(y)}$ is a diffeomorphism of $G$ 
the last  vector space is the same as 
$T_{A(x)B(y)}(G)=T_{(A\odot B)(x,y)}(G)$.  
This shows that $(x,y)$ is a regular point of $A\odot B$. 
\end{proof}

Therefore, if  $M^m$ and $N^n$ and smooth manifolds and 
$(G,\cdot)$ is a Lie group of dimension $\dim G\leq \min(m,n)$ such that 
$\varphi(M,G)$, $\varphi(N,G)$ are finite, then for any closed subgroup 
$H\subset G$ we have  $\varphi(M\times N,G/H)$ is finite and
\begin{equation}
\varphi(M\times N,G/H)\leq\varphi(M,G)\varphi(N,G).
\label{eq10.02.13.4}
\end{equation}

The right hand side inequalities follow from this inequality 
and the fact that $\varphi(S^4,S^3)=2$ (see \cite{AF1}),  
$\varphi\left(\sharp_sS^2\times S^2,S^3\right)=2s+2$ (see \cite{Fu,FPZ}). 

The left hand side inequalities follow from 
the fact that the manifolds considered do not fiber over $S^3$, by the same argument 
as in the proof of Proposition \ref{generalodd}.  In fact, we have first  $\chi(S^4\times \cdots \times S^4)=2^m$. 
Further, the Euler characteristic is almost additive, namely  
$\chi(M\sharp N)=\chi(M)+\chi(N)-(1+(-1)^n)$, for closed $n$-manifolds 
$M$ and $N$. Therefore we can compute: 
\[ \chi\left((\sharp_{r_1}S^2\times S^2)\times(\sharp_{r_2}S^2\times S^2)
\times \cdots \times (\sharp_{r_m}S^2\times S^2)\right)= 2^m(r_1+1)\cdots(r_m+1).\]

\begin{remark}
If $f:M^m\longrightarrow S^{n+1}$ ($m\geq n+1\geq 3$) is a smooth map with $r$ critical points, then one can construct a map 
$F$ with $rs$ critical points  by using fiber connected sums (see \cite{FPZ}, proof of Prop. 3.1). 
The target manifolds are of the form  $\sharp_{g}(S^1\times S^n)$.
Thus there are examples with finite $\varphi\left(M^{4m},\sharp_{g}(S^1\times S^2)\right)$. 
\end{remark}

\section{Signatures}

\subsection{Signature definition}
In order to compute the signature $\sigma(X(\Gamma))$ we need a description of the cup product $\cup$. 
Recall that for a $2n$-manifold with boundary $M$ we have $H^n(M,\partial M)\cong H_n(M)$ while 
$H^n(M,\partial M)\cong H_n(M,\partial M)^*$ by the universal coefficients theorem. 
The signature of $M$ is the one of the $(-1)^n$-symmetric 
bilinear form $\phi_M: H^n(M,\partial M)\times H^n(M,\partial M)\to \R$ given by: 
\[ \phi_M(x,y)=\langle x\cup y, [M]\rangle.\]
The adjoint of this bilinear form is the homomorphism $\phi_M: H_n(M)\to H_n(M)^*$ which 
could be identified (see e.g. \cite{Wall}) with the inclusion induced morphism in the long exact sequence: 
\[ H_n(\partial M)\to H_n(M)\stackrel{\phi_M}{\to}H_n(M,\partial M)\to H_{n-1}(\partial M)\]
 It follows that $\ker \phi_M$ is precisely the image of $H_n(\partial M)$ 
into $H_n(M)$. 

Our purpose is the explicit description of the bilinear form and its 
kernel in the case of $X(\Gamma)$.  
Unless explicitly stated we consider here $k=0$, the last subsection being concerned with the modifications to the present arguments for $k\geq 1$.

\subsection{Notation} 
Assume that we have a graph $\Gamma$ with vertices 
decorated by generalized fibered $(n-1)$-links $L_v$ 
such that $(S^{2n-1},L_v)$ are Neuwirth-Stallings pairs. 
We assume $n\geq 3$. 

On one hand $L_v$ are links of the form $K_{A_v}$ for some unimodular integer matrices $A_v$.  
The linking matrices  in the canonical framing $A^*_v$ of $K_{A_v}$  are given by Lemma \ref{KvsL}. For the sake of simplicity we assume that only white vertices which are labelled by disks can occur, whose effect is to cap off 
the boundary components. In particular we can realize trees $\Gamma$ whose leaves  are white vertices of this kind.

If $v$ is a vertex of $\Gamma$ we denote by $\Gamma_v$ the set of edges issued from $v$ and by $E(\Gamma)$ the set of 
all edges of $\Gamma$. The link $L_v$ has $d(v)$ components indexed by the edges in $\Gamma_v$. 

Let $E_v=S^{2n-1}-N(L_v)$ denote the link complement endowed with its canonical boundary trivialization. 
Thus $\partial E_v=\sqcup _{e\in \Gamma_v}(S^{n-1}\times S^{n-1})_e$, boundary components being indexed by the edges 
$e$ in $\Gamma_v$.  

%We suppose  that there are no edges of $\Gamma$  whose  both endpoints  are the same vertex.

\subsection{Homology of $X(\Gamma)$}
We have the map $f_{\Gamma}: X(\Gamma)\to D^n$, with one critical value and one singular fiber 
$V^n(\Gamma)=f_{\Gamma}^{-1}(0)$.  The retraction $D^n\to \{0\}$ lifts to a deformation 
retraction $X(\Gamma)\to V^n(\Gamma)$, so that 
\[ H_*(X(\Gamma))\cong H_*(V(\Gamma))\]
On the other hand the singular fiber $V^n(\Gamma)$ is obtained from the regular fiber $F^n$ by crushing vanishing cycles 
to points. Vanishing cycles on the local fiber $F_v$ correspond to the attaching spheres described above. Specifically, these are 
$(d(v)-1)$ embedded $(n-1)$-spheres carrying the homology of $F^n=S^n\setminus \sqcup _{i=1}^{d(v)}D_i^n$. 
The contribution of a white vertex $v$ to $V_n(\Gamma)$ is just the fiber $F_v$, which is a disk.

We can also obtain $V^n(\Gamma)$ by gluing along the pattern $\Gamma$ the local singular fibers $V_v$ which are cones along 
the boundary spheres in $\partial F_v$. Thus each edge $e$ of $\Gamma$ gives raise to a topological sphere $S^n_e\subset V_v$ obtained by suspending the sphere $S_e^{n-1}\subset \partial V_v\cong \partial F_v$ associated to the edge $e$ at 
two points corresponding to the two vertices of $e$.  Gluing together all the spheres $S^{n}_e$ by identifying the cone points 
corresponding to the same vertex  of $\Gamma$ 
we obtain $V^n(\Gamma)$. It follows that 
\[ H_*(V^n(\Gamma))\cong \left\{\begin{array}{ll}
 \oplus_{e\in E(\Gamma)}H_*(S^n_e) & {\rm for}\; *\neq 1\\
 H_1(\Gamma) &{\rm for}\; *= 1\\
 \end{array}
 \right. \]
In particular we have: 
\begin{lemma}\label{homology}
There is a preferred basis $\{\beta_e, e\in E(\Gamma)\}$ of $H_n(V(\Gamma))$ given by the $n$-cycles 
$S^n_e$. 
\end{lemma}
Note that the links $L_v$ are naturally oriented, as they bound the local fiber $F_v$. This induces a well-defined 
orientation of the $n$-cycle representing $\beta_e$.

Recall that $F$ is diffeomorphic to $\sharp_{1}^g S^1\times S^{n-1}$, where $g$ is the rank of $H_1(\Gamma)$ and thus 
$H_2(F)\cong H_{n-2}(F)=0$, if $n\neq 3$. 

Now the boundary $E=\partial X(\Gamma)$ is endowed with a fibration over $S^{n-1}=\partial D^n$ with fiber $F$. 
The Wang sequence in homology with rational coefficients gives us first:
\[ H_{n+1}(F)\to H_{n+1}(E)\to H_2(F)\]
so that $H_{n+1}(E)=0$ and by duality $H_{n-2}(E)=0$. Further the Wang sequence reads
\[0 \to H_n(F)\to H_n(E)\to H_1(F)\to H_{n-1}(F)\to H_{n-1}(E)\to H_0(F)\to 0\]

\subsection{The cup product bilinear form}\label{cupproduct} Let  $A^*_v$ denote the $d(v)\times d(v)$ linking matrix in the canonical framing of the link $L_v$. We define the matrix $A^*_{\Gamma}$ indexed by the set of edges $E(\Gamma)$:
\[ A^*(\Gamma)_{ef}= \left \{\begin{array}{ll}
(A^*_v)_{ef}, & {\rm if } \; e\cap f=v;\\
(A^*_v)_{ef}+(A^*_w)_{ef}, & {\rm if } \; e\cap f=\{v,w\};\\
0, & {\rm if } \;  e\cap f=\emptyset\\
\end{array}
\right.
\]
\begin{lemma}\label{bilinearform}
The cup product bilinear form $\phi_{X(\Gamma)}$ is expressed by the matrix $A^*(\Gamma)$ in the basis 
$\{\beta_e, e\in E(\Gamma)\}$ of $H_n(V(\Gamma))$. 
\end{lemma}
\begin{proof}
The inclusions $F\to E$, $E\to X(\Gamma)$ induce a morphism $H_n(F)\to H_n(X(\Gamma))$ whose image lies 
in the kernel of $\varphi_{X(\Gamma)}$. If we identify $H_n(X(\Gamma))$ to $H_n(V(\Gamma))$  then this has a simple description. 
Specifically, the fundamental class $[F]$ of $F$  is sent into 
$\sum_{e\in E(\Gamma)} \beta_e$. By the discussion above this element belongs to the kernel of $\phi_{X(\Gamma)}$.

Consider now two cycles $\beta_{e_1}$ and $\beta_{e_2}$ in $H_n(X(\Gamma))$. 
If $e_1\cap e_2=\emptyset$ then the intersection of these two cycles is trivial. Therefore 
\begin{equation}\label{bil1} \phi_{X(\Gamma)}(\beta_{e_1},\beta_{e_2})=0, \; {\rm if}\; e_1\cap e_2=\emptyset\end{equation}
Recall that $V^n_v\subset D^{2n}_v$ is a cone over $L_v=\sqcup_{e\in \Gamma_v} S^{n-1}_{e}$. 
Let $e_i=vw_i$, with distinct $w_i$.  There are two $n$-cycles in $D^{2n}_v$ which bound  $S^{n-1}_{e_1}$ and $S^{n-1}_{e_2}$, respectively; after putting them in general position their algebraic intersection number is $lk(S^{n-1}_{e_1}, S^{n-1}_{e_2})$  
(see e.g. \cite{Rolf}, 5.D, Ex. 9, p.134 for $n=2$). Moreover,  $S^{n-1}_{e_1}$ and $S^{n-1}_{e_2}$ also bound 
disjoint $n$-cycles in $D^{2n}_{w_1}$ and $D^{2n}_{w_2}$, respectively. Therefore we can perturb $\beta_{e_1}$ and $\beta_{e_2}$ in order to have algebraic intersection number  $lk(S^{n-1}_{e_1}, S^{n-1}_{e_2})$.  As this 
algebraic intersection number is an invariant of their homology classes we derive: 
\begin{equation}\label{bil2}\phi_{X(\Gamma)}(\beta_{e_1},\beta_{e_2})=lk(S^{n-1}_{e_1}, S^{n-1}_{e_2})=(A^*_v)_{e_1e_2}\end{equation}
Note that $A^*_v$ is the $d(v)\times d(v)$ linking matrix in the canonical framing of the link $L_v\subset \partial D_v$.

Moreover, if both edges $e\neq f$ have the same endpoints $v\neq w$ then a similar argument shows that: 
\begin{equation}\label{bil5}   \phi_{X(\Gamma)}(\beta_{e}, \beta_{f})
=(A^*_v)_{ef}+(A^*_{w})_{ef}
\end{equation}

Suppose further that $e_1=e_2$. 
If $n$ is odd, then the anti-symmetry of the bilinear form yields: 
\begin{equation}\label{bil3}\phi_{X(\Gamma)}(\beta_{e},\beta_{e})=0\end{equation}

If $n$ is even, as   $\sum_{e\in E(\Gamma)} \beta_e$ lies in the kernel of $\varphi_{X(\Gamma)}$ we derive: 
\[ \phi_{X(\Gamma)}(\sum_{e\in E(\Gamma)}\beta_e, \beta_{e_0})=\sum_{e\cap e_0\neq\emptyset}\phi_{X(\Gamma)}(\beta_e,\beta_{e_0})=0\]
Writing $e_0=vw$ we derive  
\begin{equation}\label{bil4}   \phi_{X(\Gamma)}(\beta_{e_0}, \beta_{e_0})= 
%-\sum_{e\cap e_0, e\neq e_0}\phi_{X(\Gamma)}(\beta_e,\beta_{e_0})=
-\sum_{e\in \Gamma_v\setminus \{e_0\}} (A^*_v)_{ee_0}
-\sum_{e\in \Gamma_w\setminus \{e_0\}}(A^*_w)_{e_0e}=(A^*_v)_{e_0e_0}+(A^*_{w})_{e_0e_0}\end{equation}
\end{proof}

\subsection{Geometric interpretation of the kernel of $\phi_{X(\Gamma)}$}

By induction on the number of boundary components we find: 
\[ H_i(E_v)=\left \{\begin{array}{ll}
0, & {\rm if }\; i \not\in\{0, n-1\}\\
\Q^{d(v)}, & {\rm if } \; i =n-1\\
\end{array}\right.
\]
This makes sense also when $v$ is a white vertex and hence $d(v)=1$. 

We can represent classes in $H_{n-1}(E_v)$ by means of {\em meridian} $(n-1)$-spheres on 
$\partial E_v$, which are represented as $\{p\}\times \partial D^n\subset L_v\times D^n\subset S^{2n-1}$ 
after its identification with a regular neighborhood $N(L_v)$ of $L_v$ in $S^{2n-1}$ given by the trivialization. 
We have then a preferred basis of $H_{n-1}(E_v)=\Q\langle \mu_e, e\in \Gamma_v\rangle$.

In order to compute $H_*(E)$ we will use a refined version of the Mayer-Vietoris exact sequence. 
If we have an open covering $U_i$ of $E$ such that $U_i\cap U_j\cap U_k=\emptyset$ for distinct $i,j,k$, then
the sequence below is exact:  
\[  \to\oplus_{i <j} H_k(U_i\cap U_j) \to \oplus_i H_k(U_i)\to H_k(E)\to \oplus_{i<j} H_{k-1}(U_i\cap U_j)\to \]
By taking $U_i$ to be small neighborhoods of $E_v$, we derive that 
\begin{equation} H_n(E)=\ker \left(\oplus_{e\in E(\Gamma)}H_{n-1}((S^{n-1}\times S^{n-1})_e)\to \oplus_{v}H_{n-1}(E_v)\right)\end{equation}
where the map $H_{n-1}((S^{n-1}\times S^{n-1})_e)\to \oplus_{v}H_{n-1}(E_v)$ is the morphism induced by inclusion 
if $v$ is a vertex of $e$ and zero, otherwise. 

Let us give explicit cycles for classes in   $H_n(E)$. A $(n-1)$-cycle $(z_e)\in \oplus_{e\in \Gamma_v}H_{n-1}((S^{n-1}\times S^{n-1})_e)$  
is called {\em bounding} if its image by the inclusion induced morphism in $H_{n-1}(E_v)$ vanishes.  
Thus $H_n(E)$ is identified with the space of cycles $(z_e)_{e\in E(\Gamma)}$ which restrict to 
bounding $(n-1)$-cycles on every $\Gamma_v$.  
There exists a $n$-cycle $Z_v$ in $E_v$ such that $\partial Z_v=\sum_{e\in \Gamma_v}z_e$ holds in $E_v$. 
Therefore the union $Z=\cup_{v}Z_v$ is 
a $n$-cycle in $E$ representing the class $(z_e)_{e\in E(\Gamma)}$.

Recall that $(S^{n-1}\times S^{n-1})_e$ is endowed with a canonical trivialization issued from the open book structure of $L_v$, 
namely it is foliated by the $(n-1)$-spheres  arising as intersections between $\partial E_v$ and the local fibers.  This provides a family of isotopic $(n-1)$ spheres to be called {\em preferred longitudinal} spheres, in the homology class of the canonical framing. Now 
\[ H_{n-1}(\partial E_v)=\Q\langle \lambda_e, \mu_e, e\in \Gamma_v\rangle\] 
has a basis consisting in classes of the form $\lambda_e$ which are represented by the preferred longitudinal $(n-1)$-sphere in $(S^{n-1}\times S^{n-1})_e$ and the classes of meridian spheres $\mu_e$.   We want to describe the map 
\[ i_{e,v}:H_{n-1}((S^{n-1}\times S^{n-1})_e)\to \oplus_{v}H_{n-1}(E_v)\]
in the basis defined above. 
By the definition of the meridian classes  
\begin{equation} i_{e,v}(\mu_e)=\mu_e \end{equation} 

Recall that by Hurewicz there exists an isomorphism (for $n\geq 3$)
\[ \pi_{n-1}(E_v)\to H_{n-1}(E_v;\Z)\to \Z^{\Gamma_v}\]
The image of  the class of an embedded sphere $S^{n-1}$ in $E_v$  is given by the vector 
$(lk(S^{n-1},S^{n-1}_e))_{e\in \Gamma_v}$.  Note that the image of $\mu_e$ is the vector $(\delta_{ef})_{f\in \Gamma_v}$. 
Further the preferred longitudinal spheres $\lambda_e$ and $\lambda_f$ are isotopic in $E_v$ to $S^{n-1}_e$ and $S^{n-1}_f$ respectively, so that the linking number between the corresponding embedded spheres is 
\[ lk(\lambda_e,S^{n-1}_f)= (A^*_v)_{ef}\]
Since the union of all preferred longitudinal spheres $\lambda_e$, for $e\in \Gamma_v$ 
bounds a copy of the local fiber $F_v$ we have $\sum_{e\in \Gamma_f}i_{e,v}(\lambda_e)=0$ and hence 
\[ \sum_{e\in \Gamma_v}lk(\lambda_e, S^{n-1}_f)=0\]
This yields 
\[ lk(\lambda_e,S_e^{n-1})=-\sum_{f\neq e, f\in \Gamma_v}(A^*_v)_{ef}=(A^*_v)_{ee}\] 
This proves that the image of $\lambda_e$ by the Hurewicz isomorphism is the vector 
$((A^*_v)_{ef})_{f\in \Gamma_v}$.  
Therefore 
\begin{equation}i_{e,v}(\lambda_e)=\sum_{f\in \Gamma_v} (A^*_v)_{ef}\mu_f\end{equation} 
 This identifies $\ker i_{e,v}$ to the kernel of the linear map  expressed by the matrix $ ({\bf 1} | A^*_v)$ consisting of two 
square blocks  in the basis above. Note that $A^*_v$ is not of maximal rank. 

The description of the map $H_n(E)\to H_n(V(\Gamma))$ is as follows. The retraction $r$ respects the decomposition 
of $E=\cup_{v\in \Gamma}  E_v$ and $V(\Gamma)=\cup_{v\in \Gamma}  V_v$ and it induces a commutative diagram: 
\[ \begin{array}{ccc}
H_n(E) & \to & H_n(V(\Gamma)) \\
\downarrow & & \downarrow \\
\oplus_{v\in \Gamma} H_{n-1}(\partial E_v) & \to & \oplus_{v\in \Gamma} H_{n-1}(\partial V_v)\\
\downarrow & & \downarrow \\
\oplus_{v\in \Gamma} H_{n-1}(E_v) & \to & \oplus_{v\in \Gamma} H_{n-1} (V_v)\\
\end{array}
\]

Suppose that  the class in $H_n(E)$ is given by the vector $(z_e)_{e\in E(\Gamma)}$. Then 
the retraction $r:E\to V(\Gamma)$ acts at the level of $\partial E_v$ as the parallel transport in the trivial boundary fibration 
towards  $\partial V_v$. This means that the image of $z_e=n_e\lambda_e+m_e\mu_e$ in 
$H_{n-1}(\partial V_v)$ is $n_e\lambda_e$. 
Therefore we obtain 
\begin{equation} \ker\phi_{X(\Gamma)}= \left\{\sum_{e\in E(\Gamma)}n_e\beta_e; \exists\; m_e\in \Z, {\rm such \; that}\; (n_e\lambda_e+m_e\mu_e)_e\in \ker i_{e,v}, 
\forall e\in E(\Gamma), \forall v\in e\right\}\end{equation}
Now $(n_e\lambda_e+m_e\mu_e)_e\in \ker i_{e,v}$ if and only if 
\[ m_f=-\sum_{e\in \Gamma_v} n_e (A^*_v)_{ef}\]
Let $e=vw$. When computing $i_{e,w}$ we have to note a change in orientation as $E_v$ and $E_w$ induce different orientations on their common boundary. We find that $(n_e\lambda_e+m_e\mu_e)_e\in \ker i_{e,w}$ if and only if 
\[ m_f=\sum_{e\in \Gamma_w} n_e (A^*_w)_{ef}\]
Therefore 
\[ \ker\phi_{X(\Gamma)}=\left\{\sum_{e\in E(\Gamma)}n_e\beta_e; \sum_{e\in \Gamma_v} n_e(A^*_v)_{ef}+\sum_{e\in \Gamma_w} n_e(A^*_w)_{ef}=0, 
\forall f\in \Gamma_v\cap \Gamma_w\right\}\]
This coincides indeed with the left kernel of the linear map given by $A^*(\Gamma)$. 

When $\Gamma$ is a tree $H_1(F)=0$ and the inclusion induced map $H_n(F)\to H_n(E)$ is an isomorphism.  
In this case the map $H_n(E)\to H_n(X)$  can be identified with the inclusion induced map 
$H_n(F)\to H_n(X)$. After the identification  $H_n(X)\stackrel{\simeq}{\to} H_n(V^n(\Gamma))$ the previous map is the 
same as the map induced by the retraction $H_n(F)\to H_n(V^n(\Gamma))$. In this case the kernel is one dimensional by the Wang sequence above and 
hence: 
\begin{equation}\ker \phi_{X(\Gamma)}=\Q \left\langle \sum_{e\in E(\Gamma)} \beta_e\right\rangle\end{equation}

\subsection{Signature computation when $1\leq k \leq n-2$}\label{subsignature}
For the sake of simplicity we denote $L_{\Pi^k\psi_L}$ as $L_{\Pi^kL}$.

We only consider the following cases:
\begin{enumerate}
\item the graph $\Gamma$ consists of two black vertices $v,w$ and an edge. 
The decoration is given by links of the form $L_v=L_{\Pi^kL_1}$ and $L_w=L_{\Pi^kL_2}$, obtained by the procedure
of section \ref{subhalf} from the 
$(n-1)$-dimensional generalized Hopf links $L_1$ and $L_2$  in $S^{2n-1}$;
\item the graph $\Gamma_0$ consists of one black vertex $v$ and a white vertex $w$ connected by an edge. 
The white vertex $w$ is decorated by  $F_w={\sharp_{\partial}}_{d-1}D^{n}\times S^{k}$.
The gluing along $\partial F_v$  is the identity map of 
${\sharp_{\partial}}_{d-1}S^{n-1}\times S^{k}$ and the global fiber $F$ is then 
diffeomorphic to $S^{n+k}$. 
\end{enumerate} 
We define the matrices $A^*_{\Gamma}$ and $A^*_{\Gamma_0}$ indexed by the set of edges $\{1,2,\ldots,d\}$ 
(and not by the edges of the corresponding graphs):
\[ (A^*_{\Gamma})_{ef}= 
(A^*_v)_{ef}+(A^*_w)_{ef},  \; (A^*_{\Gamma_0})_{ef}= (A^*_v)_{ef}
\]

\begin{lemma}\label{bilinearform2}
The cup product bilinear form $\phi_{X(\Gamma)}$ and $\phi_{X(\Gamma_0)}$ are expressed by the matrices 
$A^*_{\Gamma}$ and $A^*_{\Gamma_0}$, respectively in their basis 
$\{\beta_e, e\in \{1,2,\dots,d\}\}$.  
\end{lemma}
\begin{proof}
 If $E_v=S^{2n-1}\setminus N(L_v)$ 
then $\partial X(\Gamma)$, which is also denoted $E=E_v\cup E_w$,  is obtained by gluing together the link complements in a way which respects the trivialization on the boundary.
  
The generic fiber $F$ of the map $X(\Gamma)\to D^{n-k}$ is the union of the local fibers 
$F_v\cup F_w$ along the link $L_v\cong L_w\cong\sharp_{d-1}S^{n-1}\times S^k$.  
Each local fiber $F_v$ or $F_w$  is diffeomorphic to ${\sharp_{\partial}}_{d-1}S^{n-1}\times D^{k+1}$. The homeomorphism  between $F_w$ and $(S^{n}\setminus \sqcup_{i=1}^d D_i^n)\times [0,1]^k$ also provides an embedding 
of the $(n-1)$-dimensional  link $L_1=\sqcup_{e=1}^{d}S_e^{n-1}\subset F_v$. Here 
the subscript $e\in \{1,2,\ldots,d\}$ corresponds to the numbering of the spheres in the link as in the previous section. 
Their homology classes $\{\beta_e, 1\leq e\leq d\}$ generate $H_{n-1}(F)$, and according to Mayer-Vietoris we have:  
\[ H_*(F;\Q)=\left\{ \begin{array}{ll}
\Q^{d-1}, & {\rm if}\; *\in \{k+1,n-1\}\\
\Q,  & {\rm if}\; *\in \{0,n+k\}\\
0, & {\rm otherwise}. 
\end{array}
\right.
\]
where, by notation abuse for $(n,k)=(3,1)$ the set $\{k+1,n-1\}$ reduces to a singleton $\{2\}$. 
Therefore $H_{n-1}(F;\Q)$ is identified to the quotient 
\[ H_{n-1}(F;\Q)=\Q\langle\beta_e, 1\leq e\leq d\rangle/(\sum_{e=1}^d\beta_e=0)\]

As in the case $k=0$ the block $X(\Gamma)$ retracts onto the singular fiber $V(\Gamma)$ which is the 
suspension $\Sigma(\sharp_{d-1}S^{n-1}\times S^k)$ of the link $L_v$, and therefore: 
\[ H_*(X(\Gamma);\Q)=\left\{ \begin{array}{ll}
\Q^{d-1}, & {\rm if}\; *\in \{k+1,n\}\\
\Q,  & {\rm if}\; *\in\{0,n+k\}\\
0, & {\rm otherwise}. 
\end{array}
\right.
\]
Similar computations also provide: 
\[ H_*(X(\Gamma_0);\Q)=\left\{ \begin{array}{ll}
\Q^{d-1}, & {\rm if}\; *=n\\
\Q,  & {\rm if}\; *=0\\
0, & {\rm otherwise}. 
\end{array}
\right.
\]

Observe that for $k=0$ the homology of   $H_n(\Sigma(\sharp_{d-1}S^{n-1}\times S^k);\Q)=\Q^d$, according with 
the previous section. 
 
The boundary fibration  $\partial E_v\to S^{n-k-1}$, which is the restriction of the corresponding fibration of $E_v$,  
extends over $D^{n-k}$ 
and hence $\partial E_v=(\sharp_{d-1}S^{n-1}\times S^k)\times S^{n-k-1}$.  
Denote by $\lambda_e\in H_{n-1}(\partial E_v)$ the preferred longitudinal classes of the cycles  
$S_e^{n-1}\times \{pt\}\times \{pt\}$ obtained by pushing the cycles $S_e^{n-1}$ along a direction of the local fiber.
We also defines the meridian classes $\mu_e\in H_{n-1}(\partial E_v)$ as being the classes of the cycles  
$\{pt\}\times S^k_e \times S^{n-k-1}$, where $S^k_e$ is a $k$-cycle linking once $S_e^{n-1}$ and trivially 
the others $S_f^{n-1}$, for $f\neq e$ corresponding 
to the boundary of the fiber disk $D^k$ of $\Pi^k$. It is immediate that $H_{n-1}(\partial E_v)=\Q^{2(d-1)}$ and a specific basis is
deduced from the generators system below:
\[ H_{n-1}(\partial E_v)=\Q\langle\mu_e, \lambda_e; 1\leq e\leq d\rangle/(\sum_{e=1}^d \mu_e=\sum_{e=1}^d \lambda_e=0)\]
We denote by the same symbols $\mu_e$ and $\lambda_e$ the images of these classes in the homology of $E_v$. 
Using Mayer-Vietoris we deduce that 
\[ H_n(E_v)=\left\{ \begin{array}{ll}
\Q^{d-1}, & {\rm if}\;  k=n-2\\
0, & {\rm otherwise}. 
\end{array}
\right.
\]
and a description in the non-trivial case $k=n-2$ is provided by the quotient 
\[ H_{n}(E_v)=\Q\langle\lambda_e\times S^1, 1\leq e\leq d\rangle/(\sum_{e=1}^d\lambda_e\times S^1=0)\]
Further we have: 
\[ H_{n-1}(E_v)=\Q\langle\mu_e, 1\leq e\leq d\rangle/(\sum_{e=1}^d\mu_e=0)\]
Again by Mayer-Vietoris we obtain that the boundary map induces an isomorphism: for $k\neq n-2$: 
\[ H_n(E)=\ker(H_{n-1}(\partial E_v)\to H_{n-1}(E_v)\oplus H_{n-1}(E_w))\]
However, this also holds when $k=n-2$. Indeed the map $H_n(E_v)\to H_n(V(\Gamma))$ factors through $H_n(V_v)=0$. 
In order to understand $H_n(E)$ we need to describe the map 
$i_{v}:H_{n-1}(\partial E_v)\to H_{n-1}(E_v)$. It is clear that 
\[ i_{v}(\mu_e)=\mu_e\]
The inclusion map $E_v\subset  S^{2n-1}\setminus L_1$ induces a homomorphism  
$H_{n-1}(E_v)\to H_{n-1}(S^{2n-1}\setminus L_1)\cong\Q^{d}$. Its image is the subspace
$\Q^{d-1}$ of vectors whose sum vanishes. 
By the computations from the previous section we have 
\[ i_{v}(\lambda_e)=\sum_{f=1}^d (A^*_v)_{ef} \mu_f\]
Note that the right hand is well-defined in $H_{n-1}(E_v)$. 

The description of the map $H_n(E)\to H_n(V(\Gamma))$ is similar to the case $k=0$. 
 The retraction $r$ respects the decomposition 
of $E=  E_v\cup E_w$ and $V(\Gamma)=V_v\cup V_w$ and it induces a commutative diagram: 
\[ \begin{array}{ccc}
H_n(E) & \to & H_n(V(\Gamma)) \\
\downarrow & & \downarrow \\
H_{n-1}(\partial E_v)\oplus H_{n-1}(\partial E_w) & \to & H_{n-1}(\partial V_v)\oplus H_{n-1}(\partial V_w)\\
\downarrow & & \downarrow \\
 H_{n-1}(E_v)\oplus H_{n-1}(E_w) & \to & H_{n-1}(V_v)\oplus H_{n-1}(V_w)\\
\end{array}
\]

Suppose that  the class in $H_n(E)$ is given by the vector $(z_v, z_w)\in H_{n-1}(E_v)\oplus H_{n-1}(E_w)$. Then 
the retraction $r:E\to V(\Gamma)$ acts at the level of $\partial E_v$ as the parallel transport in the trivial boundary fibration 
towards  $\partial V_v$. This means that the image of $z_v=\sum_{e=1}^d n_e\lambda_e+m_e\mu_e$ in 
$H_{n-1}(\partial V_v)$ is $\sum_{e=1}^{d}n_e\lambda_e$.

Further, the arguments of section \ref{cupproduct} carry over without essential changes.  
\end{proof}

\subsection{Indefinite bilinear forms}
Since $A$ has zero diagonal, the associated bilinear form is indefinite. Now, the classification of indefinite symmetric unimodular bilinear forms over $\Z$, up to equivalence, is known. Recall that bilinear forms associated to matrices $A$ and $B$ are equivalent if there exists an invertible integral 
matrix $M$ such that $A=MBM^{\bot}$, where $M^{\bot}$ denote its transpose. Then, any indefinite unimodular  symmetric $A$
is equivalent to $pE_8\oplus q H$, for some $p,q\in \Z_+$, $q\geq 1$ (see \cite{MH}, II.5.3). Here $E_8$ denotes the Cartan matrix for the 
unimodular $E_8$ lattice and $H$ the metabolic matrix: 
\[ E_8=\left(\begin{array}{cccccccc}
2 & 1 & 0 & 0 & 0 & 0 & 0 & 0 \\
1 & 2 & 1 & 0 & 0 & 0 & 0 & 0  \\
0 & 1 & 2 & 1 & 0 & 0 & 0 & 0  \\
0 & 0 & 1 & 2 & 1 & 0 & 0 & 0   \\
0 & 0 & 0 & 1 & 2 & 1 & 0 & 1    \\
0 & 0 & 0 & 0 & 1 & 2 & 1 & 0   \\
0 & 0 & 0 & 0 & 0 & 1 & 2 & 0   \\
0 & 0 & 0 & 0 & 1 & 0 & 0 & 2   \\
\end{array}
\right), \;\; H=\left(\begin{array}{cc}
0 & 1 \\
1 & 0 \\
\end{array}\right)
\]
Let $\{e_{i,s}, i=1,8\}$ and $\{f_{1,t}, f_{2,t}\}$ be a basis for the inner product space associated to the $s$-th factor $E_8$
and the $t$-th factor $H$, respectively. The  unimodular change of basis 
\[ e'_{i,s}=e_{i,s}+f_{1,1}-f_{2,1}, \;  f'_{j,t}=f_{j,t}, \; 1\leq i\leq 8, 1\leq s\leq p, 1\leq j\leq 2, 1\leq t\leq q \] 
shows that $pE_8\oplus q H$ is equivalent to $pE'_8\oplus qH$, where: 
\[ E'_8=\left(\begin{array}{cccccccc}
0 & -1 & -2 & -2 & -2 & -2 & -2 & -2 \\
-1 & 0 & -1 & -2  & -2  & -2  & -2  & -2   \\
-2 & 1 & 0 & -1 & -2 & -2 & -2 & -2  \\
-2 & -2 & -1 & 0 & -1 & -2 & -2 & -2   \\
-2 & -2 & -2 & -1 & 0 & -1 & -2 & -1    \\
-2 & -2 & -2 & -2 & -1 & 0 & -1 & -2   \\
-2 & -2 & -2 & -2 & -2 & -1 & 0 & -2   \\
-2 & -2 & -2 & -2 & -1 & -2 & -2 & 0   \\
\end{array}
\right)
\]
Therefore unimodular symmetric matrices with zeroes on the diagonal are equivalent to  $pE'_8\oplus qH$, $q\geq 1$.

\subsection{Proof of Theorem \ref{signeven}}
If $\Gamma$ is a tree  and $k=0$ the generic global fiber is $F=S^n$. Assume that $K_A$ is fibered, e.g.  $A=\oplus_{\theta_{2n-1}}A_0$,
for some unimodular symmetric $A_0$. 
The matrix associated to $\phi_{X(\Gamma)}$ is the matrix $A^*$, which has the non-singular minor  $A^{-1}$.
Since $A$ is symmetric and unimodular $AA^{-1}A^{\bot}=A$ so that $A^{-1}$ is equivalent over $\Z$ to $A$. 
We know that $A$ is equivalent to $pE'_8\oplus q H$, for some $q\geq 1,p\geq 0$. We derive that:   
\[\sigma(X(\Gamma))=8p\]  
It suffices to consider then $A_0= p_0E'_8\oplus q_0 H$, with $p_0,q_0 \geq 1$ to obtain blocks $X(\Gamma)$ of non-zero signature. 

Novikov's additivity of the signature shows that the resulting manifold 
$M^{2n}(\Gamma)$ has signature $8p\neq 0$. In particular 
$\varphi(M^{2n}(\Gamma), S^{n})=1$ for even $n$, thereby obtaining another proof of Theorem \ref{generaleven}.

If $k\geq 1$ we consider  either of the graphs $\Gamma$ or $\Gamma_0$ from section \ref{subsignature}.
The matrix associated to $\phi_{X(\Gamma)}$ is $A^*_{\Gamma}$. 
Then, as above we take $A_v=A_w$ equivalent to $pE'_8\oplus q H$ to obtain blocks $X(\Gamma)$ and $X(\Gamma_0)$
of signature $8p\neq 0$.

Then we can consider $M(\Gamma)=X(\Gamma)\cup S^{n+k}\times D^{n-k}$. 
By Novikov's additivity of signatures $\sigma(M(\Gamma))=\sigma(X(\Gamma))$. 
By taking $A=pE'_8\oplus qH$, with $p,q\geq 1$ we obtain $\sigma(M(\Gamma))=8pm\neq 0$. 
Note that  if $n-k$ is even $M(\Gamma)$ cannot fiber over $S^{n-k}$ by the signature criterion.

\begin{remark}
Observe that gluing several such blocks $X(\Gamma_i)$ is only possible when the boundary fibrations 
$\partial X(\Gamma_i)$ are cobounding. The examples obtained in the case 
when $n$ is odd are doubles of such blocks, namely obtained by gluing $X(\Gamma)$ and $\overline{X(\Gamma)}$. Doubles of oriented manifolds are 
bounding and therefore their signatures vanishes. 
\end{remark}

\begin{remark}
All examples obtained by this procedure have signature divisible by $8\theta_{2n-1}$. We can drop 
the factor $\theta_{2n-1}$ above if we work instead of the smooth category in the topological category.  
\end{remark}

\vspace{1cm}
{\bf Acknowledgements.} The authors are grateful to R. Ara\'ujo Dos Santos, M.A.B.Hohlenwerger, O.Saeki and T.O.Souza for useful discussions. 
The first author was supported by ANR 2011 BS 01 020 01 ModGroup. Part of this work was done during his visit at University of Cluj, 
which he would like to thanks warmly for hospitality, with support from CMIRA Explora Pro 1200613701. The second author was supported through GSCE grant number 30257/22.01.2015
financed by the Babe\c s-Bolyai University.


\begin{thebibliography}{99}
\bibitem{AF1}D. Andrica, L. Funar, {\em On smooth maps with finitely many critical points}, J. London
Math. Soc. 69 (2004), 783--800, Addendum 73 (2006), 231--236. 
%\bibitem{AFK} Andrica, D., Funar, L., Kudryavtseva, E.A., {\it The minimal number of
%critical points of maps between surfaces}, Russ. J. Math. Phys. 16,
%No.3 (2009), 363-370.
%\bibitem{AMP} Andrica, D., Mangra, D., Pintea, C., {\em The circular Morse-Smale characteristic of closed surfaces}, submitted.

%\bibitem{Art}
%E. Artin, {\em Zur isotopie zweidimensionaler fl\"achen im $R^4$}, Abh. Math. Sem. Univ. Hamburg, 
%4 (1925), 174-177. 


\bibitem{ADD}
R. N. Ara\'ujo dos Santos,  Daniel Dreibelbis and N. Dutertre,  {\em Topology of the real Milnor fiber for isolated singularities, Real and Complex Singularities}, Contemporary Mathematics 569 (2012), 67-75.

\bibitem{AD}
R. N. Ara\'ujo Dos Santos and N. Dutertre, {\em Topology of real Milnor fibration for non-isolated singularities}, 
International Mathematics Research Notices (2015), DOI:10.1093/imrn/rnv286.


\bibitem{AHSS}
 R. N. Ara\'ujo Dos Santos, M.A.B.Hohlenwerger, O.Saeki and T.O.Souza, {\em New examples 
 of Neuwirth-Stallings pairs and non-trivial real Milnor fibrations},  Ann. Institut Fourier 66(2016), 83--104. 




\bibitem{At}
 M. F. Atiyah, {\em  The signature of fibre-bundles}, in "Global Analysis" (papers in honor of K. Kodaira),
73--84, Univ. Tokyo Press, Tokyo, 1969.

%\bibitem{B}
 %R. Bott, {\em The stable homotopy of the classical groups}, Annals of Math. 70(1959), 313--337. 


\bibitem{CHS}
 S. S. Chern, F. Hirzebruch, and J.-P. Serre, 
{\em On the index of a fibered manifold}, Proc. Amer. Math.
Soc.  8 (1957), 587--596.






\bibitem{Far}
 M.Farber, {\em Zeros of closed 1-forms, homoclinic orbits and Lusternik-Schnirelman theory}, Topol. Meth. Nonlinear Analysis 19 (2002), 123--152.
\bibitem{FarSchu} M. Farber,  D.Sch\"utz, {\em Closed 1-forms with at most one zero}, Topology 45 (2006), 465--473.

\bibitem{Fr} G. Friedman,  {\em Knot spinning}, chap. 4 in Handbook of Knot Theory (W. Menasco and M. Thistlethwaite, Ed.), 
Elsevier, 187--208, 2005.


\bibitem{Fu}L.  Funar,  \textit{Global classification of isolated singularities in dimensions $(4,3)$ and $(8,5)$}, 
Ann. Scuola Norm. Sup. Pisa CI. Sci. (5), Vol. X  (2011), 819--861.

\bibitem{FPZ} L. Funar, C. Pintea, and P. Zhang,  {\em Examples of smooth maps with finitely many critical
points in dimensions (4, 3), (8, 5) and (16, 9)}, Proc. Amer. Math. Soc. 138 (2010), 355--365.



\bibitem{Hae2a}
A.Haefliger, {\em  
Plongements de vari\'et\'es dans le domaine stable}, 1964, S\'eminaire 
Bourbaki 1962/1963, 
Vol. 8,  Exp. No. 245, 63--77, Soc. Math. France, Paris, 1995. 

\bibitem{Hae2b}
A.Haefliger, {\em  
Plongements diff\'erentiables dans le domaine stable}, 
Comment. Math. Helv. 37(1962/1963), 155--176. 

\bibitem{J}
 A. Jacquemard,  {\em On the fiber of the compound of a real analytic function by a projection},
Bollettino U.M.I., ser. 8, volume 2-B (1999), 263--278.

\bibitem{KN}
L. H. Kauffman and W. D. Neumann, {\em Products of knots, branched fibrations and
sums of singularities}, Topology 16 (1977), 369--393.


\bibitem{Kod}
 K. Kodaira, {\em A certain type of irregular algebraic surfaces}, J. Analyse Math. 19(1967), 207--215.


\bibitem{Lat}
F. Latour,  
{\em Existence de $1$-formes ferm\'ees non singuli\`eres dans une classe de 
cohomologie de de Rham}, 
Inst. Hautes \'Etudes Sci. Publ. Math. No. 80 
(1994), 135--194. 

\bibitem{Loo}
E.Looijenga, 
{\em A note on polynomial isolated singularities}, 
Nederl. Akad. Wetensch. Proc. Ser. A 74=Indag. Math. 33 (1971), 418--421. 


\bibitem{Michor} P.W. Michor,  Topics in differential geometry, Graduate Studies in Mathematics, Vol 93, A.M.S., 2008.

\bibitem{MH}
J. Milnor and D. Husemoller, Symmetric bilinear forms, Ergebnisse Math. 73,   Springer Verlag, 1973.

\bibitem{Rolf}
D. Rolfsen, Knots and links, AMS Chelsea Publishing AMS, Providence RI, 1990. 

\bibitem{Rud}
Lee Rudolph, {\em Isolated critical points of mappings from $\R^4$ to $\R^2$ and a natural
splitting of the Milnor number of a classical fibered link.
Part I: Basic theory; examples}, Comment. Math. Helvetici 62 (1987), 630--645.

\bibitem{MS}
A. Menegon Neto and J. Seade, {\em On the L\^e-Milnor fibration for real analytic germs}, arXiv:1604.08489.
 

%\bibitem{Stal}
%J. R. Stallings, {\em Constructions of fibred knots and links}, Proc. Symp. Pure Math XXXII,
%part 2 (1978), 55--60.

\bibitem{Ta}
F.Takens, {\em 
The minimal number of critical points of a function on a compact manifold and the Lusternik-Schnirelman category},  
Invent. Math.  6(1968), 197--244.

\bibitem{Thu}
W.P. Thurston, {\em  A norm for the homology of $3$-manifolds}, 
Mem. Amer. Math. Soc. 59 (1986), no. 339, i-vi and 99--130. 

\bibitem{Wall}
C.T.C. Wall, {\em Non-additivity of the signature}, Inventiones  Math. 7 (1969), 269--274. 

%\bibitem{Zee}
%E.C. Zeeman, {\em Twisting spun knots}, Trans. Amer. Math. Soc. 115 (1965), 471--495. 
\end{thebibliography}
\end{document}